[1994/12/01]

\documentclass[11pt]{amsart}
\usepackage{amsmath,amsthm,amsfonts,amscd,amssymb,eucal,latexsym,mathrsfs,bbm}
\usepackage{amsmath,amsthm,amsfonts,amscd,amssymb,comment,eucal,latexsym,mathrsfs}
\usepackage[all]{xy}
\newtheorem{thm}{Theorem}[section]
\newtheorem{cor}[thm]{Corollary}
\newtheorem{lem}[thm]{Lemma}
\newtheorem{prop}[thm]{Proposition}

\theoremstyle{definition}
\newtheorem{defn}[thm]{Definition}

\DeclareMathOperator{\diag}{diag}
\DeclareMathOperator{\Aut}{Aut}
\DeclareMathOperator{\OAut}{OAut}

\DeclareMathOperator{\Orb}{Orb}
\DeclareMathOperator{\Tr}{Tr}

\DeclareMathOperator{\Perm}{Perm}

\DeclareMathOperator{\alg}{alg}

\DeclareMathOperator{\Inn}{Inn}
\DeclareMathOperator{\supp}{supp}
\DeclareMathOperator{\RR}{RR}
\DeclareMathOperator{\sr}{sr}
\DeclareMathOperator{\sa}{sa}
\DeclareMathOperator{\GL}{GL}

\newcommand{\cstar}{C\ensuremath{^{*}}}

\title{Finiteness and Paradoxical Decompostions in \cstar-Dynamical Systems}
\author{Timothy Rainone}

\begin{document}

\maketitle

\begin{abstract}
We discuss the interplay between $K$-theoretical dynamics and the structure theory for certain \cstar-algebras arising from crossed products. For noncommutative \cstar-systems we present notions of minimality and topological transitivity in the $K$-theoretic framework which are used to prove structural results for reduced crossed products. In the presence of sufficiently many projections we associate to each noncommutative \cstar-system $(A,G,\alpha)$ a type semigroup $S(A,G,\alpha)$ which reflects much of the spirit of the underlying action. We characterize purely infinite as well as stably finite crossed products by means of the infinite or rather finite nature of this semigroup. We explore the dichotomy between stable finiteness and pure infiniteness in certain classes of reduced crossed products by means of paradoxical decompositions. 
\end{abstract}

\section{Introduction}
Dynamical systems and the theory of operator algebras are inextricably related~\cite{Bl2},~\cite{Kerr},~\cite{Ph}. Topological dynamics has long played a significant role in the study and classification of amenable \cstar-algebras by providing a wealth of examples that fall under the umbrella of Elliott's classification program as well as examples that lack certain regularity properties~\cite{TomsWin1}, ~\cite{TomsWin2},~\cite{GK},~\cite{EllNiu}. The crossed product construction permits the exploitation of symmetry through the acting group and is generous enough to produce a variety of \cstar-algebraic phenomena. One would like to uncover information about the the crossed product algebra by unpacking the dynamics and, conversely, describe the nature of the system by looking at the operator algebra's structure and invariants. 

Of particular interest in this paper is the deep theme common to groups, dynamical systems and operator algebras; that of finiteness, infiniteness, and proper infiniteness, the latter expressed in terms of paradoxical decompositions.  The remarkable alternative theorem of Tarski establishes, for discrete groups, the dichotomy between amenability and paradoxical decomposability. This carries over into the realm of operator algebras. Indeed, if a discrete group $\Gamma$ acts on itself by left-translation, the Roe algebra $C(\beta\Gamma)\rtimes_{\lambda}\Gamma$ is properly infinite if and only if $\Gamma$ is $\Gamma$-paradoxical and this happens if and only if $\Gamma$ is non-amenable~\cite{SR}. This is mirrored in the von Neumann algebra setting as well; all projections in a II$_1$ factor are finite and the ordering of Murray-von-Neumann subequivalence is determined by a unique faithful normal tracial state. Alternatively type III factors admit no traces since all non-zero projections therein are properly infinite. As for unital, simple, separable and nuclear algebras, the \cstar-enthusiast of old hoped  that the trace/traceless divide determined a similar dichotomy between stable finiteness and pure infiniteness (the \cstar-algebraic analog of type III). This hope was laid to rest with R{\o}rdam's example of a unital, simple, separable, nuclear \cstar-algebra containing both an infinite and a non-zero finite projection~\cite{R1}. The conjecture for such a dichotomy remains open for those algebras whose projections are total. Theorem~\ref{IntroDich} below is a result in this direction.

Despite the failure of the above dichotomy, the classification program of Elliott in its original $K$-theoretic formulation has witnessed much success for stably finite algebras ~\cite{Ro},~\cite{EllToms}, as well as in the purely infinite case with the spectacular complete classification results of Kirchberg and Phillips ~\cite{Ph2},~\cite{Kir} modulo the UCT. One motivation for studying purely infinite algebras stems from the fact that Kirchberg algebras (unital, simple, separable, nuclear, and purely infinite) are classified by their $K$- or $KK$-theory. The \cstar-literature has produced examples of purely infinite \cstar-algebras arising from dynamical systems~\cite{AD2},~\cite{KiKu},~\cite{LS},~\cite{SR}. In many cases the underlying algebra is abelian with spectrum the Cantor set. For example, Archbold, Spielberg, and Kumjian (independently) proved that there is an action of $\mathbb{Z}_{2}\ast\mathbb{Z}_{3}$ on the Cantor set so that the corresponding crossed product \cstar-algebra is isomorphic to $\mathcal{O}_2$~\cite{Sp2}. Laca and Spielberg~\cite{LS} construct purely infinite and simple crossed products that emerge from strong boundary actions. Jolissaint and Robertson~\cite{JR} generalized the idea of strong boundary action to noncommutative systems with the concept of an $n$-filling action. They showed that $A\rtimes_{\lambda}\Gamma$ is simple and purely infinite provided that the action is properly outer and $n$-filling and every corner $pAp$ of $A$ is infinite dimensional. When the algebra $A$ has a well behaved $K_{0}(A)$ group we will in fact give a $K$-theoretic proof of their result (see Proposition~\ref{JRthm}).

The transition from classical topological dynamics to noncommutative \cstar-dynamics presents several challenges and subtleties. One way to approach these issues is to interpret dynamical conditions $K$-theoretically via the induced actions on $K_{0}(A)$ and on the Cuntz semigroup $W(A)$ and use tools from the classification literature  as well as developed techniques of Cuntz comparison to uncover pertinent algebraic information. Such an approach is seen in Brown's work~\cite{B} as well as that of the author in~\cite{Ra}. We continue this philosophy here. For instance, the classical version of topological transitivity has a natural extension to noncommutative systems (Definition~\ref{TopTransDefn}), and, as in the commutative case, is tied to the primitivity of the algebra (see Theorem~\ref{ToptransPrime}). The idea of a group acting paradoxically on a set and the construction of the type semigroup goes back to the work of Tarski (the reader is encouraged to read Wagon's book~\cite{Wa} for a good treatment). R{\o}rdam and Sierakowski~\cite{SR} looked at the type semigroup $S(X,\Gamma)$ built from an action of a discrete group on the Cantor set and tied pure infiniteness of the resulting reduced crossed product to the absence of traces on this semigroup. In effect, they prove that if a countable, discrete, and exact group $\Gamma$ acts continuously and freely on the Cantor set $X$, and the preordered semigroup $S(X,\Gamma)$ is almost unperforated,  then the following are equivalent: (i) The reduced crossed product  $C(X)\rtimes_{\lambda}\Gamma$ is purely infinite, (ii) $C(X)\rtimes_{\lambda}\Gamma$  is traceless, (iii) $S(X,\Gamma)$ is purely infinite (that is $2x\leq x$ for every $x\in S(X,\Gamma)$), and (iv) $S(X,\Gamma)$ is traceless. Inspired by their work, we construct a type semigroup  $S(A,\Gamma,\alpha)$ for noncommutative systems $(A,\Gamma,\alpha)$  and establish a generalized result. This is Theorem~\ref{purelyinfinitecross} below which, in particular, implies the following.

\begin{thm}Let $A$ be a unital, separable, and exact C*-algebra with stable rank one and real rank zero. Let $\alpha:\Gamma\rightarrow\Aut(A)$ be a minimal and properly outer action with $S(A,\Gamma,\alpha)$ almost unperforated. Then the following are equivalent:
\begin{enumerate}
\item The semigroup $S(A,\Gamma,\alpha)$ is purely infinite.
\item The C*-algebra $A\rtimes_{\lambda}\Gamma$ is purely infinite.
\item The C*-algebra $A\rtimes_{\lambda}\Gamma$ is traceless.
\item The semigroup $S(A,\Gamma,\alpha)$ admits no non-trivial state.
\end{enumerate}
\end{thm}
As  a suitable quotient of $K_{0}(A)^+$, this type semigroup $S(A,\Gamma,\alpha)$ is purely infinite if and only if every positive element of $K_{0}(A)^+$ is paradoxical under the induced action with covering multiplicity at least two. Taking covering multiplicities into account, Kerr and Nowak~\cite{KN} consider completely non-paradoxical actions of a discrete group on the Cantor set. We do the same for noncommutative systems using ordered $K$-theory, and inevitably resort to Tarski's deep result (Theorem~\ref{Tarski}) to prove Theorem~\ref{stablyfinitecross}; of which the following is a special case.

\begin{thm} Let $A$ be a unital, separable and exact C*-algebra with stable rank one and real rank zero . Let  $\alpha:\Gamma\rightarrow\Aut(A)$ be a minimal action. Then the following are equivalent:
\begin{enumerate}
\item $A\rtimes_{\lambda}\Gamma$ admits a faithful tracial state.
\item $A\rtimes_{\lambda}\Gamma$ is stably finite.
\item $\alpha$ is completely non-paradoxical.
\end{enumerate}
Moreover, if $A$ is AF and $\Gamma$ is a free group, then $(1)$ through $(3)$ are all equivalent to $A\rtimes_{\lambda}\Gamma$ being MF in the sense of Blackadar and Kirchberg~\cite{BK}.
\end{thm}

Combining these two results we obtain the desired dichotomy, albeit for a certain class of crossed products.

\begin{thm}\label{IntroDich}Let $A$ be a unital, separable, and exact C*-algebra with stable rank one and real rank zero. Let $\alpha:\Gamma\rightarrow\Aut(A)$ be a minimal and properly outer action with $S(A,\Gamma,\alpha)$ almost unperforated. Then the reduced crossed product $A\rtimes_{\lambda}\Gamma$ is simple and is either stably finite or purely infinite.
\end{thm}

We round off the introduction with a brief description of the contents of this article. We begin by reviewing the necessary concepts, definitions, and results that will be assumed throughout.  In section 2 we look at minimal and topologically transitive actions with our gaze focused on the induced $K$-theoretic dynamics. Minimal actions will be characterized by a certain filling condition which will be shown to be equivalent to $K_0(A)$ admitting no non-trivial invariant order ideals (see Theorem~\ref{min3}). We then extend the notion of a topologically transitive action to noncommutative \cstar-systems  and relate such actions to the primitivity of the reduced crossed product (Theorem~\ref{ToptransPrime}). Section 3 explores the theme of finiteness and infiniteness of discrete reduced crossed products.  For any \cstar-dynamical system $(A,\Gamma,\alpha)$ we give meaning to paradoxical and completely non-paradoxical actions. We show that paradoxical type actions give rise to infinite crossed products. When the underlying algebra $A$ has a well behaved $K_{0}(A)$ group and the action is minimal and properly outer, we characterize stably finite and purely infinite discrete crossed product by means of the type semigroup $S(A,\Gamma,\alpha)$. (Theorem~\ref{stablyfinitecross}, and Theorem~\ref{purelyinfinitecross}). 

The author would like to express a deep sense of gratitude to his adviser David Kerr for his unending support. Special thanks are reserved for Adam Sierakowski and Christopher Phillips for many meaningful discussions and answered inquiries. 

\newpage

\section{Preliminaries}

Unless otherwise specified, we make the blanket assumption that all \cstar-algebras $A$ will be considered separable and with unit $1_{A}$, and all groups $\Gamma$ will be discrete. We write $\GL(A)$ for the set of invertibles in $A$, and $A_{\sa}$ for the set of  self-adjoint elements. The \cstar-algebra $A$ is of stable rank one, written $\sr(A)=1$, if $\GL(A)\subset A$ is norm-dense, and $A$ is of real rank zero, written $\RR(A)=0$, if $\GL(A)\cap A_{\sa}\subset A_{\sa}$ is norm-dense. 

This article is $K$-theoretic in flavor; the reader may want to consult~\cite{Bl} for a suitable treatment thereof, as well as~\cite{APT} for the necessary results concerning the Cuntz semigroup. We briefly outline the story-line of $K_{0}(A)$ and $W(A)$ here. If $A$ is a \cstar-algebra, $M_{m,n}(A)$ will denote the linear space of all $m\times n$ matrices with entries from $A$. The square $n\times n$ matrices $M_{n}(A)$ is a \cstar-algebra with positive cone $M_{n}(A)^+$. If $a\in M_{n}(A)^+$ and $b\in M_{m}(A)^+$, write $a\oplus b$ for the matrix $\diag(a,b)\in M_{n+m}(A)^+$.  Set $M_{\infty}(A)^+=\bigsqcup_{n\geq1} M_{n}(A)^+$; the set-theoretic direct limit of the $M_{n}(A)^+$ with connecting maps $M_{n}(A)\rightarrow M_{n+1}(A)$ given by $a\mapsto a\oplus 0$. Write $\mathcal{P}(A)$ for the set of projections in $A$ and  set $\mathcal{P}_{\infty}(A)=\bigsqcup_{n\geq1}\mathcal{P}(M_{n}(A))$. Elements $a$ and $b$ in $M_{\infty}(A)^{+}$ are said to be \emph{Pedersen-equivalent}, written $a\sim b$, if there is a matrix $v\in M_{m,n}(A)$ with $v^*v=a$ and $vv^*=b$. We say that $a$ is \emph{Cuntz-subequivalent} to (or \emph{Cuntz-smaller} than) $b$, written $a\precsim b$, if there is a sequence $(v_{k})_{k\geq1}\subset M_{m,n}(A)$ with $\|v_k^*bv_k-a\|\rightarrow 0$ as $k\rightarrow\infty$.  If $a\precsim b$ and $b\precsim a$ we say that $a$ and $b$ are \emph{Cuntz-equivalent} and write $a\approx b$. It is routine to check that $\sim$ and $\approx$ are equivalence relations on $M_{\infty}(A)^+$ and that $a\sim b$ implies $a\approx b$.  It is customary to write $V(A)=\mathcal{P}_{\infty}(A)/\sim$, and $[p]$ for the equivalence class of $p\in \mathcal{P}_{\infty}(A)$. Also set $W(A):=M_{\infty}(A)^+/\approx$  and write $\langle a\rangle$ for the class of $a\in M_{\infty}(A)^+$. $W(A)$ has the structure of a preordered abelian monoid with addition given by $\langle a\rangle+\langle b\rangle=\langle a\oplus b\rangle$ and preorder $\langle a\rangle\leq\langle b\rangle$ if $a\precsim b$. This monoid $W(A)$ will be referred to as the \emph{Cuntz semigroup} of $A$. With addition and ordering identical to that of $W(A)$, $V(A)$ is also a preordered abelian monoid. There is a cardinal difference between the orderings on $V(A)$ and $W(A)$; the ordering on $W(A)$ extends the algebraic ordering ($x,y,z\in W(A)$ with $x+y=z$ implies $x\leq z$) but only in rare cases agrees with it. With $V(A)$, the ordering agrees with the algebraic one. Indeed, one verifies that for projections $p,q\in\mathcal{P}_{\infty}(A)$, $p\precsim q$ if and only if there is a subprojection $r\leq q$ with $p\sim r$ if and only if $p\oplus p'\sim q$ for some $p'\in\mathcal{P}_{\infty}(A)$. Thus $[p]\leq[q]$ implies that $[p]+[p']=[q]$. As a brief reminder, $K_{0}(A)=\mathcal{G}(V(A))$ the Grothendieck enveloping group of $V(A)$ and $[p]_{0}=\gamma([p])$ where $\gamma:V(A)\rightarrow K_{0}(A)$ is the canonical Grothendieck map. 

A projection $p$ in $A$ is infinite if $p\sim q$ for some subprojection $q\lneq p$. It was shown in~\cite{KR2} that $p$ infinite if and only if $p\oplus b\precsim p$ for some non-zero $b\in M_{\infty}(A)^+. $ A unital \cstar-algebra $A$ is said to be infinite if $1_A$ is infinite. Otherwise, $A$ is called finite. If $M_{n}(A)$ is finite for every $n\in\mathbb{N}$ then $A$ is called stably finite.  Recall that a unital, stably finite \cstar-algebra $A$ yields an ordered abelian group $K_{0}(A)$ with positive cone $K_{0}(A)^{+}:=\gamma(V(A))$ and order unit $[1_{A}]_{0}$. Occasionally we shall require our algebras to have \emph{cancellation}, which simply means that $\gamma$ is injective. It is routine to check that algebras with stable rank one are stably finite and have cancellation. Moreover, when $A$ is stably finite the map $V(A)\rightarrow W(A)$, $[p]\mapsto\langle p\rangle$ is injective. Recall that a semigroup $K$ has the Riesz refinement property if, whenever $\sum_{j=1}^{n}x_j=\sum_{i=1}^{m}y_i$, for members $x_1,\dots,x_n,y_1,\dots,y_m\in K$, there exist $\{z_{ij}\}_{i,j}\subset K$ satisfying $\sum_{i}z_{ij}=x_j$ and $\sum_{j}z_{ij}=y_i$ for each $i$ and $j$. If $A$ is a stably finite  algebra with $\RR(A)=0$ then S. Zhang showed that $K_0(A)^+$ has the Riesz refinement property ~\cite{Zh2}.

A \emph{transformation group} is a pair $(X,\Gamma)$ where $X$ is a locally compact Hausdorff space endowed with a continuous action $\Gamma\curvearrowright X$. By a \emph{\cstar-dynamical system} we mean a triple $(A,\Gamma,\alpha)$, where $A$ is a \cstar-algebra, and $\alpha :\Gamma \rightarrow\Aut(A)$ a group homomorphism into $\Aut(A)$; the topological group of automorphisms of $A$ with the point-norm topology. In the case where $A$ is a commutative algebra, say $A=C(X)$ for some compact Hausdorff space $X$, \cstar-systems $(C(X),\Gamma,\alpha)$ are in one-to-one correspondence with transformation groups $(X,\Gamma)$ via the formula $\alpha_{s}(f)(x)=f(s^{-1}.x)$ where $s\in\Gamma, f\in C(X), x\in X$.

A \cstar-dynamical system induces a natural action at the $K$-theoretical level, and the order theoretical dynamics will reflect information about the nature of the action and will often describe the structure of the crossed product.   If $(G, G^{+},u)$ and $(H,H^{+},v)$ are ordered abelian groups each with their distinguished order units, a morphism in this category is a group homomorphism $\beta:G\rightarrow H$ which is positive and order unit preserving, i.e. $\beta(G^{+})\subset H^{+}$, and $\beta(u)=v$ respectively. We also write
\[\OAut(G):=\{\tau\in\mbox{Aut}(G): \tau(G^{+})=G^{+}, \tau(u)=u\}\]
for the set of ordered abelian group automorphisms.  Recall that if $X$ is a zero-dimensional compact metric space, $K_{0}(C(X))\cong C(X;\mathbb{Z})$ with natural point-wise ordering. The $K_{0}$-functor is covariant, namely, if $\phi:A\rightarrow B$ is a $\ast$-homomorphism ($\ast$-automorphism), one obtains a positive group homomorphism (ordered group automorphism) $K_{0}(\phi): K_{0}(A)\rightarrow K_{0}(B)$ defined by $K_{0}(\phi)([p]_{0})=[\phi(p)]_{0}$ where $p\in\mathcal{P}_{\infty}(A)$. For economy we sometimes write $\hat\phi=K_{0}(\phi)$. Note that for every action $\alpha:\Gamma \rightarrow\Aut(A)$, there is an induced action $\hat\alpha:\Gamma\rightarrow\OAut(K_{0}(A))$ where $\hat\alpha(s)= \hat\alpha_{s}:K_{0}(A)\rightarrow K_{0}(A)$ is the induced automorphism. Again, in the case of stable finiteness, the positive cone $K_{0}(A)^{+}$ is a partially ordered monoid, whose ordering is inherited from $K_{0}(A)^{+}$ and coincides with the algebraic ordering. Restricting $\hat\alpha$ to $K_{0}(A)^{+}$ also gives an action of order isomorphisms. In the same manner a \cstar-system $(A,\Gamma,\alpha)$ induces an action $\hat\alpha:\Gamma\rightarrow\OAut(W(A))$ via $\hat\alpha_{s}(\langle a\rangle)=\langle \alpha_{s}(a)\rangle$, where $s\in\Gamma$, and $a\in M_{\infty}(A)^+$. Here $\OAut(W(A))$ will denote the set of monoid isomorphisms of $W(A)$ which respect the ordering.

Given a $C^{*}$-dynamical system $(A,\Gamma,\alpha)$, we write $A\rtimes_{\alpha}\Gamma$ to denote the full crossed product \cstar-algebra whereas $A\rtimes_{\lambda,\alpha}\Gamma$ will stand for the reduced algebra (at times we will omit the $\alpha$). We briefly recall their construction and refer the reader to~\cite{BO}, \cite{Wi} and~\cite{Ph} for more details. First consider the algebraic crossed product $A\rtimes_{\alg,\alpha}\Gamma$ which is the complex linear space of all finitely supported functions $C_{c}(\Gamma,A)=\{\sum_{s\in F}a_{s}u_{s}: F\subset\Gamma, a_{s}\in A\}$, equipped with a twisted multiplication and involution: for $s,t\in\Gamma, a,b\in A$
\begin{eqnarray*}(au_{s})(bu_{t})&=&a\alpha_{s}(b)u_{st},\\
(au_{s})^{*}&=&\alpha_{s^{-1}}(a^{*})u_{s^{-1}}.
\end{eqnarray*}
If $A\subset\mathbb{B}(\mathcal{H})$ is faithfully represented (the choice of representation is immaterial), the $*$-algebra $A\rtimes_{\alg,\alpha}\Gamma$ can then be faithfully represented as operators on $\mathcal{H}\otimes\ell^{2}(\Gamma)$ via $au_{s}(\xi\otimes\delta_{t})=\alpha_{st}^{-1}(a)\xi\otimes\delta_{st}$ for $\xi\in\mathcal{H}$ and $s,t\in\Gamma$. Completing with respect to the operator norm on $\mathbb{B}(\mathcal{H}\otimes\ell^{2}(\Gamma))$ gives the reduced crossed product $A\rtimes_{\lambda,\alpha}\Gamma$. To realize the full crossed product, for each $x\in A\rtimes_{\alg,\alpha}\Gamma$, consider
\[\|x\|_{u}=\sup\|\pi(x)\|_{\mathbb{B}(\mathcal{H})}\]
where the supremum runs through all (non-degenerate) $*$-representations  $\pi:A\rtimes_{\alg,\alpha}\Gamma\rightarrow\mathbb{B}(\mathcal{H})$. Then \[A\rtimes_{\alpha}\Gamma:=A\rtimes_{\alg,\alpha}\Gamma^{-\|\cdot\|_{u}}.\]
We will at times make use of the conditional expectation $\mathbb{E}: A\rtimes_{\lambda,\alpha}\Gamma\rightarrow A$, which is a unital, contractive, completely positive map satisfying $\mathbb{E}(\sum_{s\in\Gamma}a_su_s)=a_e$.

\section{Minimality and Topological Transitivity}

In this section we develop $K$-theoretic descriptions of minimality and topological transitivity for \cstar-systems, primarily in the noncommutative setting. These formulations will be useful when describing the structure of the resulting reduced crossed product algebra.

For a general \cstar-dynamical system $\alpha:\Gamma\rightarrow\Aut(A)$, we say that $\alpha$ is \emph{minimal} (or equivalently we call $A$ $\Gamma$-\emph{simple}) if $A$ admits no non-trivial invariant ideals, that is, there does not exist an ideal $(0)\neq I\subsetneqq A$ with $\alpha_{s}(I)=I$ for every $s\in\Gamma$. Note that \emph{ideals} in the category of \cstar-algebras will always be assumed to be closed, and the term \emph{algebraic ideal} will be reserved for ideals in the algebraic sense, that is, not necessarily closed. If $A$ has a unit, it is routine to check that $A$ admits a non-trivial invariant (closed) ideal if and only if $A$ contains a non-trivial invariant algebraic ideal. Since every ideal in $M_{n}(A)$ is of the form $M_{n}(I)$ for an ideal $I\subset A$, it follows easily that if $A$ is $\Gamma$-simple, then $M_{n}(A)$ is $\Gamma$-simple as well, where the action $\Gamma\curvearrowright M_{n}(A)$ is given by amplification $s\mapsto\alpha_{s}^{(n)}\in\Aut(M_{n}(A))$.

The notion of a minimal action $\alpha:\Gamma\curvearrowright A$ is tied to the simplicity of the corresponding reduced crossed product $A\rtimes_{\lambda,\alpha}\Gamma$.  Recall that a \cstar-algebra is simple if it contains no non-trivial (closed) ideals. Indeed, given a action $\alpha:\Gamma\curvearrowright A$, with a non-trivial $\Gamma$-invariant ideal $I\subset A$, one readily sees that $I\rtimes_{\lambda,\alpha}\Gamma$ is a non-trivial ideal in $A\rtimes_{\lambda,\alpha}\Gamma$, since $(I\rtimes_{\lambda,\alpha}\Gamma)\cap A=I\neq A=(A\rtimes_{\lambda,\alpha}\Gamma)\cap A$. Therefore, a necessary condition for the reduced crossed product to be simple is minimality of the action. However, the absence of invariant ideals does not always ensure simplicity of the crossed product algebra. In some cases, however, minimality is enough to ensure a simple reduced crossed product. We record here some of the these examples.

A discrete group $\Gamma$ is said to be \emph{exact} provided that its reduced group \cstar-algebra $\cstar_{\lambda}(\Gamma)$ is exact, or equivalently, if it admits an amenable action on some compact space. Exact groups include all amenable groups and all free groups $\mathbb{F}_{r}$ for $r\in\{1,2,\dots,\infty\}$. An action $\Gamma\curvearrowright X$ is said to be \emph{free} if for each $x\in X$, the isotropy group $\{s\in\Gamma: s.x=x\}$ is trivial. It is shown in~\cite{SR} that if  $\Gamma\curvearrowright X$ is a free action of an exact group on a locally compact Hausdorff space, the reduced crossed product $C_{0}(X)\rtimes_{\lambda}\Gamma$ is simple if and only if the action is minimal.

A group $\Gamma$ is called a \emph{Powers group} if the following holds: For every finite set $F\subset \Gamma$ and integer $n\in\mathbb{N}$ there is a partition $\Gamma=E\sqcup D$ and elements $t_{1},\dots, t_{n}\in\Gamma$ such that
\begin{enumerate}
\item $sD\cap rD=\emptyset$ for every $s,r\in F$ with $s\neq r$,
\item $t_{j}E\cap t_{k}E=\emptyset$ for every $j,k\in\{1,\dots,n\}$ with $j\neq k$.
\end{enumerate}
It was shown in~\cite{Ha} that Powers' groups are non-amenable and have infinite conjugacy classes. Also, Powers showed that non-abelian free groups are Powers groups. In~\cite{HS} P. de la Harpe and G. Skandalis showed that an action $\alpha:\Gamma\rightarrow\Aut(A)$ of a Powers' group on a unital algebra $A$ is minimal if and only if $A\rtimes_{\lambda,\alpha}\Gamma$ is simple.

For general \cstar-systems $(A,\Gamma,\alpha)$, an extra condition is needed over and above minimality to ensure a simple reduced crossed product. Recall that an automorphism $\alpha$ in $\Aut(A)$ is said to be \emph{properly outer} if and only if for every invariant ideal $I\subset A$ and inner automorphism $\beta$ in $\Inn(I)$ we have $\|\alpha|_{I}-\beta\|=2$. An action $\alpha:\Gamma\rightarrow\Aut(A)$ is said to be \emph{properly outer} if for every $e\neq t\in\Gamma$, $\alpha_{t}$ is properly outer. The following result is Theorem 7.2 in~\cite{OP}.

\begin{thm} Let $(A,\Gamma,\alpha)$ be a \cstar-dynamical system with $\Gamma$ discrete and $A$ separable. If $\alpha$ is minimal and properly outer, then $A\rtimes_{\lambda,\alpha}\Gamma$ is simple.
\end{thm}

\subsection{K-Theoretic Minimality}
In the classical setting, a continuous action $\Gamma\curvearrowright X$ of a discrete group on a compact Hausdorff space is said to be \emph{minimal} if the action admits no non-trivial closed invariant sets, that is, there is no closed subset $\emptyset\neq Y\subsetneqq X$ with $s.Y=Y$ for every $s\in\Gamma$. A well known example of a minimal action is that of  an irrational rotation $\mathbb{Z}\curvearrowright\mathbb{T}$, given by $n.z=\omega^{n}z$, where $\omega=\exp(2\pi i\theta)$ for an irrational $\theta$. This, of course, agrees with the notion of a minimal action above. The equivalence of (1), (2), and (4) in the following proposition is well known  and standard in dynamics, whereas statement (3) is tailored here to serve as motivation for our work below.

\begin{prop}\label{min0}Let $\Gamma\curvearrowright X$ be a continuous action on a compact Hausdorff space, and let $\alpha:\Gamma\curvearrowright C(X)$ denote the induced action. The following are equivalent:
\begin{enumerate}
\item The action is minimal.
\item For every $x$ in $X$, the orbit $\Orb(x)=\{s.x\ |\ s\in\Gamma\}$ is dense in $X$.
\item For any non-empty open set $E\subset X$, there are elements $t_{1},\dots, t_{n}$ in $\Gamma$ such that
\[\bigcup_{j=1}^{n}t_{j}.E = X.\]
\item The $\Gamma$-algebra $C(X)$ is $\Gamma$-simple under the associated action.
\end{enumerate}
\end{prop}

\begin{proof}$(1)\Rightarrow(2)$: Fix $x\in X$, and set $Y=\overline{\Orb(x)}$. For $s\in\Gamma$, note that $s.\Orb(x)=\Orb(x)$, so taking closures we get
\[s.Y=s.\overline{\Orb(x)}=\overline{s.\Orb(x)}=\overline{\Orb(x)}=Y.\]
Since the action is minimal and $\emptyset\neq Y$, we have that $\overline{\Orb(x)}=Y=X$.

$(2)\Rightarrow(3)$: Let $\emptyset\neq E\subset X$ be open. For each finite subset $F=\{t_{1},\dots,t_{k}\}\subset\Gamma$, put $E_{F}=\cup_{j=1}^{k}t_{j}.E$. Denoting by $\mathcal{F}$ the collection of all finite sets of $\Gamma$, we claim that $\cup_{F\in\mathcal{F}}E_{F}=X$. Given the claim, compactness allows for a finite subcover $\cup_{j=1}^{J}E_{F_{j}}=X$, and thus $E_{F}=X$ where $F=\cup_{j=1}^{J}F_{j}$ which proves $(2)\Rightarrow(3)$.

To prove the claim, assume there is an $x\in X\backslash \cup_{F\in\mathcal{F}}E_{F}$. By hypothesis, $\Orb(x)$ is dense in $X$, and since $\cup_{F\in\mathcal{F}}E_{F}$ is open, there is a $z\in \cup_{F\in\mathcal{F}}E_{F}\cap\Orb(x)$. We can then write $z=s.x\in E_{F}$ for some finite set $F$ and some $s\in\Gamma$, so that $z=s.x\in t.E$ for a certain $t$, yielding $x\in(s^{-1}t).E$, a contradiction.

$(3)\Rightarrow(4)$: This direction is even easier. Suppose there is a non-trivial closed invariant set $Y$. Then $\emptyset\neq X\backslash Y=:E$ By assumption there are group elements $t_{1},\dots, t_{n}$ with $\bigcup_{j=1}^{n}t_{j}.E = X$. Thus for a point $y\in Y$, we have that $y\in t_{j}.E$ for some $j$ whence $t_{j}^{-1}.y$ belongs to $E\cap Y=\emptyset$ by invariance, which is absurd.

$(4)\Leftrightarrow(1)$: Every ideal in $C(X)$ is of the form $J_{Y}=\{f\in C(X)|\ f|_{Y}=0\}$ for some closed set $Y\subset X$. Note that $J_{Y}$ is a non-trivial and invariant if and only if $Y$ is non-trivial and invariant.
\end{proof}

An important remark on statement (3) is in order. Jolissaint and Robertson (\cite{JR}) introduced the notion of an \emph{$n$-filling action} for general \cstar-systems $(A,\Gamma,\alpha)$, which in the commutative case is equivalent to a generalized global version of hyperbolicity~\cite{LS}. More precisely, for a given integer $n\geq2$, an action $\Gamma\curvearrowright X$ of a discrete group on a compact Hausdorff space is $n$-filling if and only if for any non-empty open subsets of $X$, $E_1,\dots,E_n$, there are group elements $t_{1},\dots,t_n$ with $t_1.E_1\cup\dots\cup t_n.E_n=X$. Thus, by Proposition~\ref{min0}, an $n$-filling action is minimal. We shall see in Proposition~\ref{nfilling=nminimal} below that the $n$-filling property is equivalent to the apparently weaker condition: given any non-empty open subset $E$, there are group elements $t_{1},\dots,t_n$ with $t_1.E\cup\dots\cup t_n.E=X$. The subtle difference is that the given integer $n$ is fixed in the $n$-filling property whereas it is not necessarily bounded in Proposition~\ref{min0}.

When the space $X$ is zero-dimensional, other characterizations of minimality will be useful, indeed, they will motivate a suitable notion of $K$-theoretic minimality in the noncommutative case. Here we write $C(X;\mathbb{Z})$  for the dimension group of all continuous integer-valued functions on $X$, and $\mathcal{C}_{X}$ for the collection of all clopen subsets of a topological space $X$. The action on the underlying space induces a natural action of order automorphisms $\beta:\Gamma\rightarrow\OAut(C(X;\mathbb{Z}))$, given by $\beta_{s}(f)(x)=f(s^{-1}.x)$ for $s\in\Gamma$ and $f\in C(X;\mathbb{Z})$.

\begin{prop}\label{min1} Let $\Gamma\curvearrowright X$ be a continuous action on a compact, zero-dimensional metrizable space. Then the following are equivalent:
\begin{enumerate}
\item The action is minimal.
\item For any non-empty clopen set $E\subset X$, there are elements $t_{1},\dots, t_{n}$ in $\Gamma$ such that
\[\bigcup_{j=1}^{n}t_{j}.E = X.\]
\item For every non-zero positive function $f\in C(X;\mathbb{Z})^{+}$, there are elements $t_{1},\dots, t_{n}$ in $\Gamma$ such that
     \[\sum_{j=1}^{n}\beta_{t_{j}}(f)\geq \mathbf{1}_{X}.\]
\end{enumerate}
\end{prop}

\begin{proof}$(1)\Leftrightarrow(2)$: Identical to the proof in Proposition~\ref{min0}, use the fact that since our space is now zero-dimensional and therefore every open set (more precisely $Y^{c}$ in the proof above) contains a clopen set $E$.

$(2)\Rightarrow(3)$: Let $0\neq f\in C(X;\mathbb{Z})^{+}$. Such an $f$ has the form $f=\sum_{j=1}^{m}n_{j}\mathbf{1}_{E_{j}}$ where the $n_{j}$ are non-negative integers, not all zero, and the $E_{j}$ are clopen sets. Pick a non-empty $E_{j}:=E$ with $n_{j}\neq 0$, there is one by our assumption on $f$. Assuming  $(2)$, find elements $t_{1},\dots, t_{n}$ such that $\bigcup_{j=1}^{n}t_{j}.E = X$. Now since the $\beta_{t_{j}}$ are order preserving and $\mathbf{1}_{E}\leq f$,
\[\mathbf{1}_{X}\leq \sum_{j=1}^{n}\mathbf{1}_{t_{j}.E}= \sum_{j=1}^{n}\beta_{t_{j}}(\mathbf{1}_{E})\leq \sum_{j=1}^{n}\beta_{t_{j}}(f).\]

$(3)\Rightarrow(2)$: Given a non-empty clopen set $E$, $f:=\mathbf{1}_{E}$ is a non-negative, non-zero, integer-valued continuous function. We then are granted group elements $t_{1},\dots, t_{n}$ in $\Gamma$ such that $\sum_{j=1}^{n}\beta_{t_{j}}(f)\geq \mathbf{1}_{X}$. Then
\[\mathbf{1}_{X}\leq \sum_{j=1}^{n}\beta_{t_{j}}(f)=\sum_{j=1}^{n}\beta_{t_{j}}(\mathbf{1}_{E})=\sum_{j=1}^{n}\mathbf{1}_{t_{j}.E},\]
which shows $\bigcup_{j=1}^{n}t_{j}.E = X$.

\end{proof}

Recall that when $X$ is the Cantor set, $K_{0}(C(X))$ is order isomorphic to $C(X;\mathbb{Z})$ via the dimension map $\dim:K_{0}(C(X))\rightarrow C(X;\mathbb{Z})$ given by $\dim([p]_0)(x)=\Tr(p(x))$. Here $p$ represents a projection over the matrices of $C(X)$; $M_n(C(X))\cong C(X;\mathbb{M}_{n})$, and $\Tr$ denotes the standard (non-normalized) trace on $\mathbb{M}_{n}$. Now given a continuous action $\Gamma\curvearrowright X$, let $\alpha:\Gamma\rightarrow\Aut(C(X))$ denote the associated action on the algebra $C(X)$, and write $\hat\alpha:\Gamma\rightarrow\OAut(K_{0}(C(X)))$ for the induced action on the ordered group $K_{0}(C(X))$. Moreover, as above, we have a natural action $\beta:\Gamma\rightarrow\OAut(C(X;\mathbb{Z}))$, given by $\beta_{s}(f)(x)=f(s^{-1}.x)$ for $s\in\Gamma$ and $f\in C(X;\mathbb{Z})$. One may inquire about the possible equivariance of $\hat\alpha$ and $\beta$ through the isomorphism $\dim$. Indeed, these actions are the same; we show that for each $s\in\Gamma$, the following diagram is commutative.

$$\begin{CD}
K_{0}(C(X)) @>\hat\alpha_{s}>>K_{0}(C(X)) \\
@VV\dim V  @VV\dim V\\
C(X;\mathbb{Z}) @>\beta_{s}>> C(X;\mathbb{Z})
\end{CD}$$
To see this, consider any projection $p\in\mathcal{P}_{\infty}(C(X))$, any $s\in\Gamma$ and any $x\in X$. We compute:

\begin{align*}
\beta_{s}\circ\dim([p]_0)(x)&=\dim([p]_0)(s.^{-1}x)=\Tr(p(s.^{-1}x))=\Tr(\alpha_{s}(p)(x))=\dim([\alpha_{s}(p)]_0)(x)\\&=\dim\circ\hat\alpha_{s}([p]_0)(x)
\end{align*}
which shows that $\beta_{s}\circ\dim([p]_0)=\dim\circ\hat\alpha_{s}([p]_0)$ as functions on $X$, and consequently that $\beta_{s}\circ\dim=\dim\circ\hat\alpha_{s}$ by uniqueness of the Grothendieck extension.\\

Condition $(3)$ in the above Propostion and this discussion motivate a suitable definition for \emph{minimal} actions at the $K$-theoretic level in the noncommutative case, at least for stably finite algebras where the $K_{0}$ group is ordered.

\begin{defn}Let $\Gamma$ be a discrete group, $A$ a unital, stably finite \cstar-algebra, and $\alpha:\Gamma\curvearrowright A$ an action with induced action $\hat\alpha$ on $K_{0}(A)$.
\begin{enumerate}
\item  We say that $\alpha$ is \emph{$K_{0}$-minimal} provided that for every $0\neq g\in K_{0}(A)^{+}$, there are $t_{1},\dots,t_{n}$ in $\Gamma$ such that $\sum_{j=1}^{n}\hat\alpha_{t_{j}}(g)\geq [1]_{0}$.
\item Fix an integer $n\in\mathbb{N}$. We say that $\alpha$ is \emph{$K_{0}$-$n$-minimal} provided that for every $0\neq g\in K_{0}(A)^{+}$, there are $t_{1},\dots,t_{n}$ in $\Gamma$ such that $\sum_{j=1}^{n}\hat\alpha_{t_{j}}(g)\geq [1]_{0}$.
\item Fix an integer $n\in\mathbb{N}$. We say that $\alpha$ is \emph{$K_{0}$-$n$-filling} provided that for all non-zero $g_1,\dots,g_n\in K_{0}(A)^{+}$, there are $t_{1},\dots,t_{n}$ in $\Gamma$ such that $\sum_{j=1}^{n}\hat\alpha_{t_{j}}(g_j)\geq [1]_{0}$.
\end{enumerate}
\end{defn}

There is a significant difference between $K_0$-minimal actions and $K_{0}$-$n$-minimal actions. Of course every $K_0$-$n$-minimal action is $K_0$-minimal, but the converse is far from true. We shall see that when $K_0(A)$ has suitable properties $K_0$-$n$-minimal actions along with proper outerness guarantee that the reduced crossed product is simple and purely infinite, whereas $K_0$-minimal actions along with proper outerness may generate simple stably finite crossed product algebras. 

Proposition~\ref{min1} and the remarks proceeding it imply that a Cantor system $(X,\Gamma)$ is minimal if and only if the algebra $C(X)$ is $\Gamma$-simple if and only if $\alpha$ is $K_{0}$-minimal, where $\alpha:\Gamma\curvearrowright C(X)$ is, of course, the induced action. With some work, we will show that for a stably finite algebra that admits sufficiently many projections, $K_{0}$-minimality and $\Gamma$-simplicity are equivalent notions. Due to the rigid structure of $K_0$, it turns out to be easier to work with the Cuntz semigroup $W(A)$. Also, when dealing with Cuntz comparability we need not make any restrictions on the underlying algebra. Here are the parallel definitions.

\begin{defn}Let $\Gamma$ be a discrete group, $A$ a unital \cstar-algebra, and $\alpha:\Gamma\curvearrowright A$ an action with induced action $\hat\alpha$ on the Cuntz semigroup $W(A)$.
\begin{enumerate}
\item  We say that $\alpha$ is \emph{$W$-minimal} provided that for every $0\neq g\in W(A)$, there are $t_{1},\dots,t_{n}$ in $\Gamma$ such that $\sum_{j=1}^{n}\hat\alpha_{t_{j}}(g)\geq \langle1\rangle$.
\item Fix an integer $n\in\mathbb{N}$. We say that $\alpha$ is \emph{$W$-$n$-minimal} provided that for every $0\neq g\in W(A)$, there are $t_{1},\dots,t_{n}$ in $\Gamma$ such that $\sum_{j=1}^{n}\hat\alpha_{t_{j}}(g)\geq \langle1\rangle$.
\item Fix an integer $n\in\mathbb{N}$. We say that $\alpha$ is \emph{$W$-$n$-filling} provided that for all non-zero $g_1,\dots,g_n\in W(A)$, there are $t_{1},\dots,t_{n}$ in $\Gamma$ such that $\sum_{j=1}^{n}\hat\alpha_{t_{j}}(g_j)\geq \langle 1\rangle$.
\end{enumerate}
\end{defn}

Using topological transitivity we show below (Proposition~\ref{nfilling=nminimal}) that $W$-$n$-minimal and $W$-$n$-filling actions coincide. But first, we justify our choice of nomenclature.

\begin{prop}\label{min2}Let $(A,\Gamma,\alpha)$ be a C*-dynamical system with induced action $\hat\alpha:\Gamma\curvearrowright W(A)$ on the Cuntz semigroup of $A$. Then $A$ is $\Gamma$-simple if and only if $\alpha$ is $W$-minimal.
\end{prop}

\begin{proof} Suppose the action is $W$-minimal and let $(0)\neq I\subset A$ be a $\Gamma$-invariant ideal. Take a nonzero $x$ in $I^{+}$ and find group elements $t_{1},\dots,t_{n}$ with $\sum_{j=1}^{n}\hat\alpha_{t_{j}}(\langle x\rangle)\geq \langle1\rangle$. This means
\[\langle\alpha_{t_{1}}(x)\oplus\dots\oplus\alpha_{t_{n}}(x)\rangle=\sum_{j=1}^{n}\langle\alpha_{t_{j}}(x)\rangle
=\sum_{j=1}^{n}\hat\alpha_{t_{j}}(\langle x\rangle)\geq \langle1\rangle\approx \langle1\oplus0_{n-1}\rangle. \]
This implies that $1\oplus0_{n-1}$ is Cuntz smaller than $\alpha_{t_{1}}(x)\oplus\dots\oplus\alpha_{t_{n}}(x)$ and there is a sequence $(y_{k})_{k\geq1}$ in $M_{n}(A)$ with $y_{k}^{*}(\alpha_{t_{1}}(x)\oplus\dots\oplus\alpha_{t_{n}}(x))y_{k}\rightarrow1\oplus0_{n-1}$. Now each $\alpha_{t_{j}}(x)$ belongs to $I$ so that $\alpha_{t_{1}}(x)\oplus\dots\oplus\alpha_{t_{n}}(x)$ belongs to $M_{n}(I)$, a (closed) ideal in $M_{n}(A)$. Furthermore, each $y_{k}^{*}(\alpha_{t_{1}}(x)\oplus\dots\oplus\alpha_{t_{n}}(x))y_{k}\in M_{n}(I)$ so that $1\oplus0_{n-1}$ lives in $M_{n}(I)$ ($M_{n}(I)$ is closed) which implies that $1\in I$ and $I=A$. The action is thus $\Gamma$-simple.

Conversely, assume $\alpha$ admits no non-trivial invariant ideals, and let $g=\langle a\rangle\in W(A)$, for some $a\in M_{n}(A)^{+}$. Since the algebraic ideal generated by $\{\alpha_{s}^{(n)}(a) : s\in\Gamma\}$ is all of $M_{n}(A)$, there are lists of elements $t_{1},\dots,t_{m}\in\Gamma$, and $x_{1},\dots,x_{m}; y_{1},\dots,y_{m}$ in $M_{n}(A)$ such that
\[\sum_{j=1}^{m}x_{j}\alpha_{t_{j}}^{(n)}(a)y_{j}^{*}=\frac{1}{2}\textbf{1}_{M_{n}(A)}.\]
Now set $z_{j}:=x_{j}+y_{j}$ and observe that
\begin{align*}\sum_{j=1}^{m}z_{j}\alpha_{t_{j}}^{(n)}(a)z_{j}^{*}&=\sum_{j=1}^{m}x_{j}\alpha_{t_{j}}^{(n)}(a)y_{j}^{*}+\sum_{j=1}^{m}y_{j}\alpha_{t_{j}}^{(n)}(a)x_{j}^{*}
+\sum_{j=1}^{m}x_{j}\alpha_{t_{j}}^{(n)}(a)x_{j}^{*}+\sum_{j=1}^{m}y_{j}\alpha_{t_{j}}^{(n)}(a)y_{j}^{*}\\
&\geq \sum_{j=1}^{m}x_{j}\alpha_{t_{j}}^{(n)}(a)y_{j}^{*}+\big(\sum_{j=1}^{m}x_{j}\alpha_{t_{j}}^{(n)}(a)y_{j}^{*}\big)^{*}=\textbf{1}_{M_{n}(A)}\geq 1_{A}\oplus0_{m-1},
\end{align*}
the first inequality following from the fact that the last two sums on the first line are positive. A simple Cuntz comparison now gives
\begin{align*}1_{A}\approx1_{A}\oplus0_{m-1}&\precsim\sum_{j=1}^{m}z_{j}\alpha_{t_{j}}^{(n)}(a)z_{j}^{*}
=(z_{1},\dots,z_{m})(\alpha_{t_{1}}(a)\oplus\dots\oplus\alpha_{t_{m}}(a))(z_{1},\dots,z_{m})^{*}\\
&\precsim\alpha_{t_{1}}(a)\oplus\dots\oplus\alpha_{t_{m}}(a).
\end{align*}
Therefore, in the ordering on $W(A)$,
\[\langle1\rangle\leq \langle\alpha_{t_{1}}(a)\oplus\dots\oplus\alpha_{t_{m}}(a)\rangle=\sum_{j=1}^{m}\langle\alpha_{t_{j}}(a)\rangle
=\sum_{j=1}^{m}\hat\alpha_{t_{j}}(\langle a\rangle)\]
which gives the $W$-minimality of the action.
\end{proof}

It is well known that if a \cstar algebra $A$ is unital and stably finite, $(K_{0}(A),K_{0}(A)^{+},[1]_0)$ is a well ordered abelian group with order unit $u=[1]_0$, and so the above definition of $K_{0}$-minimality applies. With the added assumption of sufficiently many projections, all the notions of minimality mentioned above will coincide as the next result shows. Recall that a subgroup $H$ of an abelian ordered group $(G,G^{+})$ is said to be an \emph{order ideal} provided that its positive cone is spanning and hereditary, that is, $H=H^{+}-H^{+}$ and $0\leq g\leq h\in H^{+}$ implies $g\in H$, where by definition $H^{+}=H\cap G^{+}$. In the context of an action $\beta:\Gamma\rightarrow\OAut(G)$, a subset $H\subset G$ is called $\Gamma$-invariant if for every $t\in\Gamma$, $\beta_{t}(H)\subset H$.

\begin{thm}\label{min3} Let $A$ be a unital, stably finite C*-algebra with the property that every ideal in $A$ admits a non-trivial projection. Consider an action $\alpha:\Gamma\rightarrow\Aut(A)$ with induced action $\hat\alpha:\Gamma\rightarrow\OAut(K_{0}(A))$. The following are equivalent:
\begin{enumerate}
\item $A$ is $\Gamma$-simple.
\item $\alpha$ is $W$-minimal.
\item $\alpha$ is $K_{0}$-minimal.
\item There are no non-trivial $\Gamma$-invariant order ideals $H\subset K_{0}(A)$.
\end{enumerate}
\end{thm}

\begin{proof} The equivalence of (1) and (2) was shown in Propostion~\ref{min2}.

$(2)\Rightarrow(3)$: Let $0\neq x\in K_{0}(A)^{+}$, then $x=[p]_{0}$ for some non-zero $p\in\mathcal{P}_{m}(A)$.  By hypothesis there are group elements $t_{1},\dots, t_{n}$ such that
\[\langle\alpha_{t_{1}}(p)\oplus\dots\oplus\alpha_{t_{n}}(p)\rangle=\sum_{j=1}^{n}\langle\alpha_{t_{j}}(p)\rangle
=\sum_{j=1}^{n}\hat\alpha_{t_{j}}(\langle p\rangle)\geq \langle 1\rangle.\]
By definition $1\precsim r:=\alpha_{t_{1}}(p)\oplus\dots\oplus\alpha_{t_{n}}(p)$ and so $1\sim q\leq r$ where $q$ is a subprojection of $r$ in $M_{mn}(A)$. Since $r-q\perp q$, a small computation will give the desired inequality, indeed:
\begin{align*}[1]_{0}\leq [1]_0+[r-q]_0&=[q]_0+[r-q]_0=[r-q+q]_0=[r]_0=[\alpha_{t_{1}}(p)\oplus\dots\oplus\alpha_{t_{n}}(p)]_0\\
&=\sum_{j=1}^{n}\hat\alpha_{t_{j}}([p]_0)=\sum_{j=1}^{n}\hat\alpha_{t_{j}}(x).
\end{align*}

$(3)\Rightarrow(1)$: Suppose $(0)\neq I\subset A$ is a $\Gamma$-invariant ideal. By our assumption on $A$, we can find a nonzero projection $p\in I$. Now find group elements $t_{1},\dots, t_{n}$ such that
\[[\alpha_{t_{1}}(p)\oplus\dots\oplus\alpha_{t_{n}}(p)]_0=\sum_{j=1}^{n}[\alpha_{t_{j}}(p)]_0
=\sum_{j=1}^{n}\hat\alpha_{t_{j}}([p]_0)\geq [1]_0.\]
Apply the order embedding $V(A)\hookrightarrow W(A)$ which gives $\langle\alpha_{t_{1}}(p)\oplus\dots\oplus\alpha_{t_{n}}(p)\rangle\geq\langle 1\rangle$ so that
\[1\oplus 0_{n-1}\approx 1\precsim \alpha_{t_{1}}(p)\oplus\dots\oplus\alpha_{t_{n}}(p).\]
Now follow the exact reasoning as Propostion~\ref{min2} to deduce that $1\in I$ and $I=A$.

$(3)\Rightarrow(4)$: Suppose $(0)\neq H\subset K_{0}(A)$ is a $\Gamma$-invariant order ideal. Since $H^{+}$ is spanning, we can locate a non-zero $x$ in $H^{+}:=H\cap K_{0}(A)^{+}$. By (3) there are $t_{1},\dots,t_{n}\in\Gamma$ with $\sum_{j=1}^{n}\hat\alpha_{t_{j}}(x)\geq [1]_0\geq 0$. Each $\hat\alpha_{t_{j}}(x)$ is in $H$ so $\sum_{j=1}^{n}\hat\alpha_{t_{j}}(x)\in H$, and $H$ being hereditary implies that $[1]_0$ is in $H$. Now given any $z\in K_{0}(A)^{+}$, there is an $n\in\mathbb{Z}^{+}$ such that $0\leq z\leq n[1]_0$. Using again the fact that $H$ is hereditary we have $z\in H$, thus $K_{0}(A)^{+}\subset H$ whence $K_{0}(A)=H$.

$(4)\Rightarrow(3)$: Let $0\neq x\in K_{0}(A)^{+}$. Consider the set
\[L:=\bigg\{y\in K_{0}(A) : \exists n\in\mathbb{Z}^{+}, \exists t_{1},\dots,t_{n}\in\Gamma\quad\mbox{such that}\quad 0\leq y\leq\sum_{j=1}^{n}\hat\alpha_{t_{j}}(x)\bigg\}.\]
Two facts are fairly clear about $L\subset K_{0}(A)^{+}$: $L+L\subset L$ and $L$ is hereditary, that is, if $z\in K_{0}(A)$ and $y\in L$ with $0\leq z\leq y$ then $z\in L$. It is natural to then define the subgroup $H=L-L$. We show that $H$ is in fact a non-zero $\Gamma$-invariant order ideal. To that end set $H^{+}=H\cap K_{0}(A)^{+}$ and note that $L\subset H^{+}$. Then
\[H=L-L\subset H^{+}-H^{+}\subset H,\]
so $H=H^{+}-H^{+}$. Also, if $z\in K_{0}(A)$ with $0\leq z\leq y-y'\in H$, with $y,y'\in L$, then since $y-y'\leq y$ and $L$ is hereditary, we have $z\in L\subset H$ so $H$ is hereditary as well. $H\neq(0)$ since $x\in H$. Finally, if $y\in L$ and $t\in\Gamma$, then $0\leq y\leq \sum_{j=1}^{n}\hat\alpha_{t_{j}}(x)$ for certain group elements $t_{1},\dots, t_{n}$. Applying the order isomorphism $\hat\alpha_{t}$ we get
\[0\leq\hat\alpha_{t}(y)\leq\hat\alpha_{t}\big(\sum_{j=1}^{n}\hat\alpha_{t_{j}}(x)\big)=\sum_{j=1}^{n}\hat\alpha_{tt_{j}}(x),\]
which implies that $\hat\alpha_{t}(y)\in L$ and $\hat\alpha_{t}(L)\subset L$. So $\hat\alpha_{t}(H)=\hat\alpha_{t}(L)-\hat\alpha_{t}(L)\subset L-L=H$ which is what we wanted. By our hypothesis, $H=K_{0}(A)$, so that $[1]_0\in H$. Writing $[1]_0=y-y'\leq y$ for some $y,y'$ in $L$ and recalling that $L$ is hereditary ensures $[1]_0\in L$, which means that there are group elements $t_{1},\dots,t_{n}$ with $\sum_{j=1}^{n}\hat\alpha_{t_{j}}(x)\geq [1]_0$, and $\alpha$ is thus $K_{0}$-minimal.

\end{proof}

\subsection{K-Theoretic Topological Transitivity}

We now aim to develop a notion of topological transitivity in the noncommutative setting, and we do this using $K$-theory. An action $\Gamma\curvearrowright X$ of a group on a locally compact Hausdorff space is termed \emph{topologically transitive} if for every pair $U,V$ of non-empty open subsets of $X$, there is a group element $s\in\Gamma$ with $s.U\cap V\neq\emptyset$. When $X$ is compact, it is routine to check that every minimal action is topologically transitive (see Proposition~\ref{min0} above) but the converse is false in general as witnessed by the translation action $\mathbb{Z}\curvearrowright\mathbb{Z}_{\infty}$ on the one-point compactification of the integers with the point $\infty$ being fixed. An action $\Gamma\curvearrowright X$ is said to have the \emph{intersection property} if each non-zero ideal of $C_{0}(X)\rtimes_{\lambda}\Gamma$ has non-zero intersection with $C_{0}(X)$. As minimality of an action is linked with simplicity of the crossed product, topological transitivity is associated with primitivity. The following is an abbreviated form of Proposition 2.8 of~\cite{MR}. 

\begin{prop}Consider an continuous action of a discrete group on a locally compact Hausdorff space $X$. If $C_{0}(X)\rtimes_{\lambda}\Gamma$ is prime, then the action is topologically transitive. Conversely, if the action is topologically transitive and has the intersection property, then $C_{0}(X)\rtimes_{\lambda}\Gamma$ is prime. 
\end{prop}

After we develop a notion of topological transitivity in the noncommutative setting we will establish a more general result (see Theorem~\ref{ToptransPrime}).

\begin{defn}\label{TopTransDefn} Let $(A,\Gamma,\alpha)$ be a \cstar-system. Call an action $\alpha$ \emph{topologically transitive} if for every pair of non-zero $x,y\in W(A)$, there is group element $t\in\Gamma$ and a non-zero $z\in W(A)$ with $z\leq x$ and $\hat\alpha_{t}(x)\leq y$.
\end{defn}

The following Proposition shows that this definition is consistent with the established notion of topological transitivity in the commutative setting. Recall that for $f,g\in M_{\infty}(C(X))^+$, we have $f\precsim g$ if and only if $\supp(f)\subset\supp(g)$, where $\supp(\cdot)$ denotes the support.

\begin{prop} Let $X$ be a locally compact space, and let $k:\Gamma\curvearrowright X$ be a continuous action with induced action $\alpha:\Gamma\rightarrow\Aut(C_0(X))$. Then $k$ is topologically transitive if and only if $\alpha$ is topologically transitive.
\end{prop}

\begin{proof}
Assume that that $k$ is topologically transitive, and let $x=\langle g\rangle$, $y=\langle f\rangle$ be non-zero elements in $W(C_0(X))$. Since $f$ and $g$ are continuous matrix valued functions on $X$, $U=\{x|f(x)\neq0\}$ and $V=\{x|g(x)\neq0\}$ are open and non-empty. Therefore, there is a $s\in\Gamma$ such that $s.U\cap V\neq\emptyset$. Consider any non-empty open subset $Y\subset s.U\cap V$ and find a non-zero continuous function $h:X\rightarrow[0,1]$ with $\supp(h)\subset Y$. Since $\supp(h)\subset V\subset\supp(g)$ we have that $h\precsim g$ whence $0\neq z:=\langle h\rangle\leq\langle g\rangle=x$. Also,
\[\supp(s^{-1}.h)=s^{-1}.\supp(h)\subset s^{-1}.Y\subset s^{-1}.(s.U)=U\subset\supp(f),\]
thus $s^{-1}.h\precsim f$ which gives $\hat\alpha_{s^{-1}}(z)=\langle s^{-1}.h\rangle\leq\langle f\rangle=y$.

Conversely, now suppose $\alpha:\Gamma\curvearrowright C_0(X)$ is topologically transitive and consider a pair $U,V$ of non-empty open subsets of $X$. Find continuous non-zero mappings $f,g:X\rightarrow[0,1]$ with $\supp(f)\subset U$ and $\supp(g)\subset V$. There is then a non-zero $z\in W(C_0(X))$ and $t\in\Gamma$ with $z\leq\langle f\rangle$ and $\hat\alpha_{t}(z)\leq\langle g\rangle$. Say $z=\langle h\rangle$ for some continuous $h\in C_0(X,\mathbb{M}_{n}^{+})$. Then $\supp{h}\subset\supp(f)\subset U$ and $t.\supp(h)=\supp(t.h)\subset\supp(g)\subset V$. Now set $Y:=\{x | h(x)\neq0\}$, a non-empty open set and observe that $\emptyset\neq Y\subset\supp(h)\subset U\cap t^{-1}.V$.
\end{proof}

As in the commutative case, every minimal action is topologically transitive.

\begin{prop} Let $A$ be a unital C*-algebra. If $\alpha:\Gamma\rightarrow\Aut(A)$ is a minimal action, then it is topologically transitive.
\end{prop}
\begin{proof}
Let $x,y\in W(A)$ be non-zero, without loss of generality we may assume $x=\langle a\rangle$ and $y=\langle b\rangle$ with $a,b\in A^+$. By minimality there are group elements $t_1,\dots, t_n\in\Gamma$ with
\[\langle\alpha_{t_1}(a)\oplus\cdots\oplus\alpha_{t_n}(a)\rangle=\sum_{j=1}^{n}\langle\alpha_{t_j}(a)\rangle
=\sum_{j=1}^{n}\hat\alpha_{t_j}(\langle a\rangle)=\sum_{j=1}^{n}\hat\alpha_{t_j}(x)\geq\langle1\rangle.\]
There is a sequence $(v_k)_{k\geq1}$ in $M_{n\times1}(A)$ with $v_{k}^*(\alpha_{t_1}(a)\oplus\cdots\oplus\alpha_{t_n}(a))v_k\rightarrow1$ in $A$ as $k\rightarrow\infty$. For each $k$ write $v_k=(v_{k,1},\dots,v_{k,n})^T$ so that
\[\bigg(\sum_{j=1}^{n}v_{k,j}^*\alpha_{t_{j}}(a)v_{k,j}\bigg)_{k\geq1}\longrightarrow1,\quad\mbox{as}\quad k\rightarrow\infty.\]
With $k$ large enough we have $\big\|1-\sum_{j=1}^{n}v_{k,j}^*\alpha_{t_{j}}(a)v_{k,j}\big\|<1/2$. There is a $y\in A$ with
\[(1_{A}-1/2)_+=y^*\bigg(\sum_{j=1}^{n}v_{k,j}^*\alpha_{t_{j}}(a)v_{k,j}\bigg)y,\]
which gives $1_A=\sum_{j=1}^{n}u_j^*\alpha_{t_{j}}(a)u_j$ where $u_j=2^{1/2}v_{k,j}y$. It follows that for every $j=1,\dots,n$
\[b=\sum_{j=1}^{n}b^{1/2}u_j^*\alpha_{t_{j}}(a)u_jb^{1/2}\geq b^{1/2}u_j^*\alpha_{t_{j}}(a)u_jb^{1/2}\geq0.\]
Choose an $i$ such that $b^{1/2}u_i^*\alpha_{t_{i}}(a)u_ib^{1/2}\neq0$ (there is one since $b\neq0$), and set \[c=\alpha_{t_{i}^{-1}}(b^{1/2}u_i^*\alpha_{t_{i}}(a)u_ib^{1/2})=(\alpha_{t_{i}^{-1}}(u_ib^{1/2}))^*a\alpha_{t_{i}^{-1}}(u_ib^{1/2}).\] Then $c\neq0$, $c\precsim a$ and $\alpha_{t_{i}}(c)\precsim b$. With $z:=\langle c\rangle$, we have $z\leq x$ and $\hat\alpha_{t_{i}}(z)\leq y$ so $\alpha$ is topologically transitive.
\end{proof}

Recall that a \cstar-algebra $B$ is \emph{prime} if for every pair of non-trivial ideals $I,J\subset B$, $IJ=I\cap J\neq(0)$. It natural to ask what dynamical conditions give rise to prime reduced crossed products. We briefly study this issue.

A \cstar-system $(A,\Gamma,\alpha)$ is said to have the \emph{intersection property} if every ideal $I\subset A\rtimes_{\lambda,\alpha}\Gamma$ has non-trivial intersection with $A$. If the action $\alpha$ is properly outer, then the intersection property follows (see lemma~\ref{CuntzSmaller}). When $A=C_0(X)$, proper outerness is equivalent to topological freeness, and it well known that if the action is topologically free, the reduced crossed product $C_0(X)\rtimes_{\lambda}\Gamma$ is prime if and only if the action $\Gamma\curvearrowright X$ is topologically transitive. We now can generalize this to the noncommutative setting.

\begin{thm}\label{ToptransPrime} Let $A$ be a \cstar-algebra, $\Gamma$ a countable discrete group and $\alpha:\Gamma\rightarrow\Aut(A)$ an action. If $A\rtimes_{\lambda,\alpha}\Gamma$ is prime then $\alpha$ is topologically transitive. Conversely, if $(A,\Gamma,\alpha)$ has the intersection property and $\alpha$ is topologically transitive then $A\rtimes_{\lambda,\alpha}\Gamma$ is prime.
\end{thm}

\begin{proof}
Assume $\alpha$ is topologically transitive and that $(A,\Gamma,\alpha)$ has the intersection property. Let $I$ and $J$ be non-zero ideals in $A\rtimes_{\lambda,\alpha}\Gamma$. By the intersection property there are $0\neq x\in I\cap A$ and $0\neq y\in J\cap A$. Set $a=x^*x\in I\cap A^+$ and $b=y^*y\in J\cap A^+$. By topological transitivity there is a $0\neq z\in W(A)$ and $t\in\Gamma$ with $z\leq\langle a\rangle$ and $\hat\alpha_{t}(z)\leq\langle b\rangle$. Writing $z=\langle c\rangle$ for some $c\in M_{n}(A)^+$, we have $c\precsim a$ and $\alpha_{t}(c)\precsim b$. There is a sequence $v_k\in M_{1\times n}(A)$ with $v_k^*av_k\rightarrow c$ as $k\rightarrow\infty$. If $v_k=(v_{k,1},\dots,v_{k,n})$, and $c=(c_{i,j})_{i,j}$ then for every $1\leq i,j\leq n$ we get $v_{k,i}^*av_{k,j}\rightarrow c_{i,j}$ as $k\rightarrow\infty$. Note that $v_{k,i}^*av_{k,j}\in A\cap I$ for each $i,j,k$, so $c_{i,j}\in A\cap I$ for every $i,j$. Since $A\cap I$ is a $\Gamma$-invariant ideal in $A$ we know that $\alpha_{t}(c_{i,j})\in A\cap I$ for every $i,j$.

Similarly there is a sequence $u_k\in M_{1\times n}(A)$ with $u_k^*bu_k\rightarrow \alpha_{t}(c)$ as $k\rightarrow\infty$ giving $u_{k,i}^*bu_{k,j}\rightarrow \alpha_{t}(c_{i,j})$ for every $i,j$ where $u_k=(u_{k,1},\dots,u_{k,n})$. Since $u_{k,i}^*bu_{k,j}$ belongs to $A\cap J$ for every $i,j$ so do the $\alpha_{t}(c_{i,j})$. With $c$ non-zero, there is a $c_{i,j}\neq0$ so that
$\alpha_{t}(c_{i,j})\in(A\cap I)\cap (A\cap J)\subset I\cap J$. Thus $A\rtimes_{\lambda}\Gamma$ is prime.

Conversely, now suppose $A\rtimes_{\lambda}\Gamma$ is prime. Let $x,y\in W(A)$ be nonzero. We can write $x=\langle a\rangle$ and $y=\langle b\rangle$ with $a,b\in M_{n}(A)^+$. Since $A\rtimes_{\lambda}\Gamma$ is prime, $M_{n}(A\rtimes_{\lambda}\Gamma)\cong M_{n}(A)\rtimes_{\lambda,\alpha^(n)}\Gamma$ is prime, so we can find a non-zero $c\in M_{n}(A)$ and $s\in\Gamma$ with
\[0\neq b^{1/2}cu_{s^{-1}}a^{1/2}=b^{1/2}cu_{s^{-1}}a^{1/2}u_su_{s^{-1}}=b^{1/2}c\alpha_{s^{-1}}(a^{1/2})u_{s^{-1}}.\]
Multiplying on the right by the unitary $u_{s}$, we get $v:=b^{1/2}c\alpha_{s^{-1}}(a^{1/2})$ is non-zero in $M_{n}(A)$. Setting $w=\alpha_s(v)$ we get $z:=\langle ww^*\rangle\leq\langle a\rangle=x$ since
\[ww^*=\alpha_s(v)\alpha_s(v)^*=\alpha_{s}(b^{1/2}c)a^{1/2}(\alpha_{s}(b^{1/2}c)a^{1/2})^*=\alpha_{s}(b^{1/2}c)a\alpha_{s}(b^{1/2}c)^*\precsim a.\]
On the other hand,
\begin{align*}w^*w&=\alpha_s(v)^*\alpha_s(v)=\alpha_{s}(v^*v)=\alpha_s(\alpha_{s^{-1}}(a^{1/2})c^*bc\alpha_{s^{-1}}(a^{1/2}))
\\&=a^{1/2}\alpha_{s}(c)^*\alpha_{s}(b)\alpha_{s}(c)a^{1/2}=(\alpha_{s}(c)a^{1/2})^*\alpha_{s}(b)\alpha_{s}(c)a^{1/2}\precsim\alpha_{s}(b),
\end{align*}
which says that $z=\langle ww^*\rangle=\langle w^*w\rangle\leq\langle\alpha_{s}(b)\rangle=\hat\alpha_{s}(\langle b\rangle)=\hat\alpha_{s}(y)$. Therefore we have found $0\neq z\in W(A)$, and $t:=s^{-1}\in\Gamma$ with $z\leq y$ and $\hat\alpha_{t}(z)\leq y$ as was required.
\end{proof}

We end with a cute result that will be needed later on.

\begin{prop}\label{nfilling=nminimal} Let $(A,\Gamma,\alpha)$ be a \cstar-system. Then $\alpha$ is $W$-$n$-minimal if and only if $\alpha$ is $W$-$n$-filling. 
\end{prop}

\begin{proof} The $n$-minimal property easily follows from the $n$-filling property. For the converse, let $x_1,\dots, x_n$ be non-zero in $W(A)$. Let $t_1=e$. Since $\alpha$ is $n$-minimal, $\alpha$ is topologically transitive, so we can find $0\neq z_1\leq x_1$ and $t_2\in\Gamma$ with $z_1\leq t_2.x_2$.  Next, again by transitivity find $0\neq z_2\leq z_1$ and $t_3\in\Gamma$ with $z_2\leq t_3.x_3$. We continue in this fashion until we find $0\neq z_{n-1}\leq z_{n-2}$ and $t_n\in\Gamma$ with $z_{n-1}\leq t_n.x_n$. Now apply the $n$-minimal property to locate $s_1,\dots,s_n$ in $\Gamma$ with $\sum_{j=1}^{n}s_j.z_{n-1}\geq\langle 1\rangle$. From these orderings we get
\begin{align*}s_1t_n.x_n&\geq s_1.z_{n-1}, \quad s_2t_{n-1}.x_{n-1}\geq s_2.z_{n-2}\geq s_2.z_{n-1},\dots\\
\dots, s_{n-1}t_2.x_2&\geq s_{n-1}.z_1\geq s_{n-1}.z_{n-1},\quad s_n.x_1\geq s_n.z_1\geq s_n.z_{n-1}.
\end{align*}
We thus obtain
\[\sum_{j=1}^{n}s_{n-j+1}t_j.x_j\geq\sum_{j=1}^{n}s_j.z_{n-1}\geq\langle 1\rangle\]
so that $\alpha$ is indeed $W$-$n$-filling.
\end{proof}

\section{Finiteness, Paradoxical Decompositions,  and the Type Semigroup}

In this section we study $K$-theoretic conditions, in the form of paradoxical phenomena, that characterize finite and infinite crossed products.  As a brief reminder, a projection $p\in A$ is properly infinite if there are two subprojections $q,r\leq p$ with $qr=0$ and $q\sim p\sim r$.  The algebra $A$ is properly infinite if $1_A$ is properly infinite.  If every hereditary \cstar-subalgebra of $A$ contains a properly infinite projection then $A$ is called purely infinite. S. Zhang showed that $A$ is purely infinite if and only in $\RR(A)=0$ and every projection in $A$ is properly infinite~\cite{Zh1}. It was a longstanding open question whether there existed a unital, separable, nuclear and simple \cstar-algebra which was neither stably finite or purely infinite. M. R{\o}rdam settled the issue in~\cite{R1} by exhibiting a unital, simple, nuclear, and separable \cstar-algebra $D$ containing a finite and infinite projection $p,q$. It follows that $A=qDq$ is unital, separable, nuclear, simple, and properly infinite, but not purely infinite. It is natural to ask if there is a smaller class of algebras for which such a dichotomy exists. Theorem~\ref{Di} below is a result in this direction.

\subsection{Paradoxical Decompositions}

We first construct infinite algebras arising from crossed products by generalizing the notion of a local boundary action to the noncommutative setting. A continuous action $\Gamma\curvearrowright X$ of a discrete group on a locally compact space is called a \emph{local boundary action} if for every non-empty open set $U\subset X$ there is an open set $V\subset U$ and $t\in\Gamma$ with $t.\overline{V}\subsetneq V$. Laca and Spielberg showed in~\cite{LS} that such actions yield infinite projections in the reduced crossed product $C_{0}(X)\rtimes_{\lambda}\Gamma$. Sierakowski remarked that the condition $t.\overline{V}\subsetneq V$ for \emph{some} non-empty open set $V$ and group element $t\in\Gamma$ is equivalent to the existence of open sets $U_1, U_2\subset X$ and elements $t_1, t_2\in\Gamma$ such that $U_1\cup U_2=X$, $t_1.U_1\cap t_2.U_2=\emptyset$, and $t_1.U_1\cup t_2.U_2\neq X$. He generalized this by defining \emph{paradoxical} actions. A transformation group $(X,\Gamma)$ is $n$-paradoxical if there exist open subsets $U_1,\dots, U_n\subset X$ and elements $t_1,\dots, t_n\in\Gamma$ such that
\[\bigcup_{j=1}^{n}U_{j}=X, \qquad \bigsqcup_{j=1}^{n}t_j.U_{j}\subsetneq X.\]
He then showed that the algebra $C(X)\rtimes_{\lambda}\Gamma$ is infinite provided that $X$ is compact and the action $\Gamma\curvearrowright X$ is $n$-paradoxical for some $n$. We do the same here in the noncommutative setting.

Let $\alpha:\Gamma\rightarrow\Aut(A)$ be a \cstar-dynamical system where $\Gamma$ is a discrete group. Once again, we look at the induced actions $\hat\alpha:\Gamma\curvearrowright K_{0}(A)^{+}$ and $\hat\alpha:\Gamma\curvearrowright W(A)$ given by $t.x=\hat\alpha_{t}(x)$ for $t\in\Gamma$ and $x\in K_{0}(A)^{+}$ or $W(A)$.

\begin{prop}Let $A$ be a stably finite C*-algebra with cancellation and such that $K_{0}(A)^+$ has Riesz refinement. Let $\alpha:\Gamma\rightarrow\Aut(A)$ be a $K_{0}$-paradoxical action in the sense that there exist $x_{1},\dots, x_{n}\in K_{0}(A)^+$ and group elements $t_{1},\dots, t_{n}\in\Gamma$ with 
\[\sum_{j=1}^{n}x_j\geq [1_{A}]_0,\quad \mbox{and}\quad \sum_{j=1}^{n}\hat\alpha_{t_{j}}(x_j)< [1_{A}]_0.\] 
Then $A\rtimes_{\lambda}\Gamma$ is infinite.
\end{prop}
 
 \begin{proof}Denote by $\iota: A\rightarrow A\rtimes_{\lambda}\Gamma$ the canonical embedding. Given that $\sum_{j=1}^{n}x_j\geq [1_{A}]_0$, there is an $r\in\mathcal{P}_{\infty}(A)$ with $\sum_{j=1}^{n}x_j= [1_{A}]_0+[r]_0$.  With the refinement property one can find elements $\{y_{j}\}_{j}^{n}, \{z_{j}\}_{j}^{n}\subset K_{0}(A)^{+}$ with \[ x_j=y_j+z_j,\quad \sum_{j=1}^{n}y_j=[1_{A}]_0,\quad \sum_{j=1}^{n}z_j=[r]_0 .\]
Then $\sum_{j=1}^{n}\hat\alpha_{t_{j}}(y_j)\leq\sum_{j=1}^{n}\hat\alpha_{t_{j}}(x_j)< [1_{A}]_0$. Cancellation implies there are mutually orthogonal projections $p_{1},\dots, p_{n}$ in $A$ with $[p_j]_0=y_j$, as well as mutually orthogonal  projections $q_{1},\dots, q_{n}$ in $A$ with $[q_j]_0=\hat\alpha_{t_{j}}(x_j)=\hat\alpha_{t_{j}}([p_j]_0)=[\alpha_{t_{j}}(p_j)]_0$. It also implies that $q_j\sim \alpha_{t_{j}}(p_j)$ as projections in $A$ for each $j$, whence $\iota(q_j)\sim\iota(\alpha_{t_{j}}(p_j))\sim\iota(p_j)$ as projections in $A\rtimes_{\lambda}\Gamma$. Setting $p=\sum_{j}p_j$, and $q=\sum_{j}q_j$ we obtain
\[\iota(q)=\sum_{j=1}^{n}\iota(q_j)\sim \sum_{j=1}^{n}\iota(p_j)=\iota(p).\]
On the other hand $[p]_0=\big[\sum_{j}p_j\big]=\sum_{j}[p_j]_0=\sum_{j}y_j=[1_A]_0$. Cancellation once more implies $p\sim 1_A$ and therefore $\iota(p)\sim\iota(1_A)=1_{A\rtimes_{\lambda}\Gamma}$. Thus we have $\iota(q)\sim 1_{A\rtimes_{\lambda}\Gamma}$. 

All is needed to show is that $\iota(q)\neq 1_{A\rtimes_{\lambda}\Gamma}$. To this end we observe that
\[[q]_0=\big[\sum_{j=1}^{n}q_j\big]_0=\sum_{j=1}^{n}[q_j]_0=\sum_{j=1}^{n}\hat\alpha_{t_{j}}(x_j)<[1]_0,\]
so that $[1_A]_0-[q]_0=[1_A-q]_0\neq0$, which implies $1-q\neq0$ by stable finiteness. Therefore $\iota(q)\neq\iota(1_A)=1_{A\rtimes_{\lambda}\Gamma}$ and $A\rtimes_{\lambda}\Gamma$ is infinite as claimed.
 \end{proof} 
 
A similar result holds with less restrictions on the underlying algebra $A$ but with a slight strengthening on the dynamics. For this result we will make the following convention: for $x,y\in W(A)$ we shall write $x<y$ to mean $x+z\leq y$ for some non-zero $z\in W(A)$.

\begin{prop} Let $A$ be a unital C*-algebra and let $\alpha:\Gamma\rightarrow\Aut(A)$ be an action which is $W$-paradoxical in the sense that there exist $x_{1},\dots, x_{n}\in W(A)$ and group elements $t_{1},\dots, t_{n}\in\Gamma$ with $\sum_{j=1}^{n}x_j\geq \langle1_{A}\rangle$ and $\sum_{j=1}^{n}\hat\alpha_{t_{j}}(x_j)< \langle1_{A}\rangle$. Then $A\rtimes_{\lambda}\Gamma$ is infinite.
\end{prop}

\begin{proof} Again let $\iota: A\rightarrow A\rtimes_{\lambda}\Gamma$ denote the canonical embedding and for $t\in\Gamma$ write $u_t$ for the canonical unitary in $A\rtimes_{\lambda}\Gamma$ that implements the action $\alpha_t: A\rightarrow A$, so that $\iota(\alpha_{t}(a))=u_t\iota(a)u_t^*\approx \iota(a)$ for every $a\in A$ and $t\in\Gamma$. If $a\in M_{n}(A)^+$ then by amplification we have  $\iota^{(n)}(\alpha_{t}^{(n)}(a))=(u_t\otimes 1_A)\iota^{(n)}(a)(u_t\otimes 1_{n})^*\approx \iota^{(n)}(a)$ for every $t\in\Gamma$. For economy we will omit denoting the amplification when the context is understood.

For each $j=1,\dots,n$ set $x_j=\langle a_j\rangle$ for $a_{j}\in M_{\infty}(A)^+$. Then we have 
\[\langle1_A\rangle\leq \sum_{j=1}^{n}x_j=\sum_{j=1}^{n}\langle a_j\rangle=\langle a_1\oplus\hdots \oplus a_n\rangle,\]
which implies $1_A\precsim \oplus_{j=1}^{n}a_j$ in $M_{\infty}(A)^+$. Applying $\iota$ we get 
$1_{A\rtimes_{\lambda}\Gamma}\precsim \oplus_{j=1}^{n}\iota(a_j)\approx \oplus_{j=1}^{n}\iota(\alpha_{t_{j}}(a_j))$ in  $M_{\infty}(A\rtimes_{\lambda}\Gamma)^+$. 

By our convention we have
\[\langle\alpha_{t_{1}}(a_1)\oplus\hdots\oplus\alpha_{t_{n}}(a_{n})\oplus b \rangle=\sum_{j=1}^{n}\hat\alpha_{t_{j}}(x_j)+\langle b\rangle\leq\langle1_{A}\rangle\]
for some non-zero $b\in M_{\infty}(A)^+$. Thus $\alpha_{t_{1}}(a_1)\oplus\hdots\oplus\alpha_{t_{n}}(a_{n})\oplus b\precsim 1_A$ and $\iota(\alpha_{t_{1}}(a_1))\oplus\hdots\oplus\iota(\alpha_{t_{n}}(a_{n}))\oplus \iota(b)\precsim 1_{A\rtimes_{\lambda}\Gamma}$. Together we get
\[1_{A\rtimes_{\lambda}\Gamma}\oplus\iota(b)\precsim\iota(\alpha_{t_{1}}(a_1))\oplus\hdots\oplus\iota(\alpha_{t_{n}}(a_{n}))\oplus \iota(b)\precsim 1_{A\rtimes_{\lambda}\Gamma}.\]
Since $\iota(b)=\neq0$ it follows that $A\rtimes_{\lambda}\Gamma$ is infinite as claimed.
\end{proof}

We make the brief remark that an action $\Gamma\curvearrowright A$ is $K_0$-paradoxical in the above sense with $n=2$  if and only if there is a non-zero $x\in \Sigma(A)$ (the scale of $A$) and $t\in\Gamma$ with $\hat\alpha_{t}(x)<x$. 

Perhaps what has been called \emph{paradoxical} is misleading because, in a sense, paradoxicality implies the idea of duplication of sets. Gleaning from the ideas explored in~\cite{KN}, we define a notion of paradoxical decomposition with covering multiplicity in the noncommutative setting.

\begin{defn} Let $A$ be a \cstar-algebra, $\Gamma$ a discrete group and $\alpha:\Gamma\rightarrow\Aut(A)$ an action with its induced action $\hat\alpha$. Let $0\neq x\in K_{0}(A)^{+}$ and $k>l>0$ be positive integers. We say $x$ is \emph{$(\Gamma,k,l)$-paradoxical} if there are $x_{1},\dots,x_{n}$ in $K_{0}(A)^{+}$ and $t_{1},\dots,t_{n}$ in $\Gamma$ such that
\[\sum_{j=1}^{n}x_{j}\geq kx,\qquad \mbox{and}\qquad \sum_{j=1}^{n}\hat\alpha_{t_{j}}(x_{j})\leq lx.\]

If an element $x\in K_{0}(A)^{+}$ fails to be $(\Gamma,k,l)$-paradoxical for all integers $k>l>0$ we call $x$ \emph{completely non-paradoxical}. The action $\alpha$ will be called \emph{completely non-paradoxical} if every member of $K_{0}(A)^{+}$ is completely non-paradoxical.
\end{defn}

The notion of a quasidiagonal action was first introduced in ~\cite{KN} and further studied in~\cite{Ra} from a $K$-theoretic viewpoint. The author of~\cite{Ra} observed that MF (or equivalently QD) actions of discrete groups $\Gamma$ on AF algebras admit, in a local sense, $\Gamma$-invariant traces on $K_{0}(A)$, so it should come to no surprise that these actions do not allow paradoxical decompositions at the $K$-theoretic level. The next proposition illustrates this principle and provides us with our first class of examples of completely non-paradoxical actions.

\begin{prop}\label{Para1} If $\alpha:\Gamma\rightarrow\Aut(A)$ is an MF action of a discrete group $\Gamma$ on a unital AF algebra, then $\alpha$ is completely non-paradoxical.
\end{prop}

\begin{proof} Suppose $0\neq x\in K_{0}(A)^{+}$ is $(\Gamma,k,l)$-paradoxical for some positive integers $k>l>0$, so that there are $x_{1},\dots,x_{n}$ in $K_{0}(A)^{+}$ and $t_{1},\dots,t_{n}$ in $\Gamma$ such that
\[y:=\sum_{j=1}^{n}x_{j}\geq kx\qquad \mbox{and}\qquad z:=\sum_{j=1}^{n}\hat\alpha_{t_{j}}(x_{j})\leq lx.\]
Consider the finite sets $F=\{t_{1},\dots,t_{n}\}\subset\Gamma$, and $S=\{y-kx, lx-z,x_{1},\dots,x_{n},x\}\subset K_{0}(A)^{+}$. Since $\alpha$ is quasidiagonal, Proposition 4.8 of~\cite{Ra} guarantees existence of a subgroup $H\leq K_{0}(A)$ which contains all the $F$-iterates of $S$, and a group homomorphism $\beta:H\rightarrow\mathbb{Z}$ with $\beta(\hat\alpha_{t}(g))=\beta(g)$ for each $t\in F$ and $g\in S$. Also, $\beta(g)>0$ for $0<g\in S$. Clearly $y,z,kx,lx$ all belong to the subgroup $H$, and since $\beta(y-kx)\geq 0$, we have $k\beta(x)=\beta(kx)\leq \beta(y)$. Similarly, $\beta(z)\leq l\beta(x)$. Now using the $\Gamma$-invariance of $\beta$,
\[k\beta(x)\leq\beta(y)=\beta\big(\sum_{j=1}^{n}x_{j}\big)=\sum_{j=1}^{n}\beta(x_{j})=\sum_{j=1}^{n}\beta(\hat\alpha_{t_{j}}(x_{j}))=\beta\big(\sum_{j=1}^{n}\hat\alpha_{t_{j}}(x_{j})\big)
=\beta(z)\leq l\beta(x).\]
This is absurd since $\beta(x)>0$ and $l<k$. Thus no such non-zero $x$ exists.
\end{proof}

It was shown by Kerr and Nowak~\cite{KN} that quasidiagonal actions by groups whose reduced group algebras are MF give rise to MF crossed products, which are always stably finite. Indeed, it is the finiteness of the crossed product that is an obstruction to a positive element being paradoxical.

\begin{prop}\label{SF implies CNP} Consider a C*-dynamical system $(A,\Gamma,\alpha)$ with stably finite reduced crossed product $A\rtimes_{\lambda}\Gamma$. Then the induced $\hat\alpha:\Gamma\curvearrowright K_{0}(A)^{+}$ is completely non-paradoxical.
\end{prop}

\begin{proof} Suppose on the contrary that $0\neq [p]_0:=x\in K_{0}(A)^{+}$ is $(\Gamma,k,l)$ paradoxical for some integers $k>l>0$ where $p\in\mathcal{P}_{m}(A)$. We then have elements $x_{1},\dots, x_{n}$ in $K_{0}(A)^{+}$ and $t_{1},\dots, t_{n}\in\Gamma$ with
\[\sum_{j=1}^{n}x_{j}\geq kx\qquad \mbox{and}\qquad \sum_{j=1}^{n}\hat\alpha_{t_{j}}(x_{j})\leq lx.\]

If $\iota: A\hookrightarrow A\rtimes_{\lambda}\Gamma$, $\iota:a\mapsto au_{e}$, denotes the canonical embedding, apply $\hat\iota:K_{0}(A)^{+}\rightarrow K_{0}(A\rtimes_{\lambda}\Gamma)^{+}$ which is order preserving to obtain
\[k\hat\iota(x)=\hat\iota(kx)\leq\hat\iota\big(\sum_{j=1}^{n}x_{j}\big)=\sum_{j=1}^{n}\hat\iota(x_{j})
=\sum_{j=1}^{n}\hat\iota\hat\alpha_{t_{j}}(x_{j})=\hat\iota\big(\sum_{j=1}^{n}\hat\alpha_{t_{j}}(x_{j})\big)
\leq\hat\iota(lx)=l\hat\iota(x).\]
Here we used the fact that for a projection $q$ in $A$ and $s\in\Gamma$ we have
\begin{align*}
\hat\iota([q]_{K_{0}(A)})&=[\iota(q)]_{K_{0}(A\rtimes_{\lambda}\Gamma)}=[u_squ_s^*]_{K_{0}(A\rtimes_{\lambda}\Gamma)}
=[\alpha_{s}(q)]_{K_{0}(A\rtimes_{\lambda}\Gamma)}\\&=\hat\iota[\alpha_{s}(q)]_{K_{0}(A)}=\hat\iota\hat\alpha_{s}([q]_{K_{0}(A)}),
\end{align*}
so that $\hat\iota=\hat\iota\hat\alpha_{s}$ agree as maps $K_{0}(A)^{+}\rightarrow K_{0}(A\rtimes_{\lambda}\Gamma)^{+}$.

The fact that $A\rtimes_{\lambda}\Gamma$ is stably finite now implies that $\hat\iota(x)=0$,which means that $\iota(p)=0$, so $p=0$, a contradiction.
\end{proof}

\subsection{The type semigroup}

We wish to establish a converse to Proposition~\ref{SF implies CNP}. For this we shall need more machinery. Analogous to the type semigroup of a general group action (see~\cite{Wa}), we associate to each suitable \cstar-system $(A,\Gamma,\alpha)$  a preordered abelian monoid $S(A,\Gamma,\alpha)$ which correctly reflects the above notion of paradoxicality in $K_{0}(A)$, and then inevitably resort to Tarski's result (Theorem~\ref{Tarski} below) tying the existence of states on $S(A,\Gamma,\alpha)$ to non-paradoxicality. We embark on the details.

Let us first recall the notion of equidecomposability for group actions and the construction of the type semigroup. Suppose a group $\Gamma$ acts on a set $X$, and let $\mathcal{C}$ be a $\Gamma$-invariant subalgebra of the power set $\mathcal{P}(X)$. Orthogonality is then built in as we enlarge the action as follows. Let $Y=X\times\mathbb{N}_0$, and $G=\Gamma\times\Perm(\mathbb{N}_0)$ where $\mathbb{N}_0=\mathbb{N}\cup\{0\}$. We then have a canonical action $G\curvearrowright Y$ given by
\[(t,\sigma).(x,n)=(t.x,\sigma(n)).\]
For a set $E\subset Y$, and $j\in\mathbb{N}_{0}$ the \emph{jth level} of $E$ is the set $E_j=\{x\in X : (x,j)\in E\}$. We say that $E$ is \emph{bounded} if only finitely many levels $E_j$ are non-empty. Now consider the algebra of $G$-invariant  subsets
\[S(X,\mathcal{C})=\{E\subset Y : E\quad\text{is bounded and}\quad E_j\in\mathcal{C},\ \forall j\in\mathbb{N}_{0}\}.\]
Subsets $E,F\in S(X,\mathcal{C})$ are said to be \emph{$G$-equidecomposable}, and we write $E\sim_{G}F$, if there are $E_{1},\dots,E_{n}\in S(X,\mathcal{C})$, and $g_{1},\dots,g_{n}\in G$ such that:
\[E=\bigsqcup_{j=1}^{n}E_{j},\quad\text{and}\quad F=\bigsqcup_{j=1}^{n}g_j.E_{j}.\]
The notation $\sqcup$ is used to emphasize the fact that the partitioning sets are disjoint. Reflexivity and symmetry of the relation $\sim_{G}$ are straightforward, and transitivity follows from taking refined partitions. We quotient out by the equivalence relation $\sim_{G}$, setting
\[S(X,\Gamma,\mathcal{C}):=S(X,\mathcal{C})/\sim_{G},\]
and write $[E]$ for the equivalence class of $E\in S(X,\mathcal{C})$. Addition is then defined on classes via
\[\bigg[\bigcup_{j=1}^{n}E_j\times\{j\}\bigg]+\bigg[\bigcup_{i=1}^{m}F_i\times\{i\}\bigg]
=\bigg[\bigcup_{j=1}^{n}E_j\times\{j\}\cup \bigcup_{i=1}^{m}F_i\times\{n+j\}\bigg].\]
A little work shows that addition is well defined and $[\emptyset]$ is a neutral element. Endowed with the algebraic ordering, $S(X,\Gamma,\mathcal{C})$ has the structure of a preordered abelian monoid, often referred to as the type semigroup~\cite{Wa}.

We aim to construct a similar monoid for noncommutative \cstar-systems $(A,\Gamma,\alpha)$, at least in the presence of sufficiently many projections. The philosophy is that elements of the positive cone $K_{0}(A)^+$ would represent our ``subsets'' as it were, and the idea of refined partitions is reflected by suitable refinement properties displayed in the additive structure of $K_{0}(A)^{+}$. If we are to translate the notion of equidecomposability to the $K_{0}$-setting, we shall require that $A$ be an algebra for which the monoid $K_{0}(A)^{+}$ has the the Riesz refinement property. This discussion thus motivates the following definition.

\begin{defn} Let $A$ be a  $C^{*}$-algebra, $\Gamma$ a discrete group, and let $\alpha:\Gamma\rightarrow\Aut(A)$ an action. We define a relation on $K_{0}(A)^{+}$ as follows:
\[x\sim_{\alpha} y\quad (x,y\in K_{0}(A)^{+})\]
\[\Longleftrightarrow \]
\[\exists\  \{u_{j}\}_{j=1}^{k}\subset K_{0}(A)^{+}, \ \{t_{j}\}_{j=1}^{k}\subset\Gamma,\quad\mbox{such that}\quad \sum_{j=1}^{k}u_{j}=x\quad\mbox{and}\quad \sum_{j=1}^{k}\hat\alpha_{t_{j}}(u_{j})=y.\]
\end{defn}

\begin{lem} If $A$ is a stably finite C*-algebra such that $K_{0}(A)^+$ has the Riesz refinement property, then $\sim_{\alpha}$ as defined above is an equivalence relation.
\end{lem}

\begin{proof} Let $x,y\in K_{0}(A)^{+}$. Clearly $x\sim_{\alpha} x$, simply take $u_{1}=x$ and $t_{1}=e$. If $x\sim_{\alpha} y$, via the decomposition $x=\sum_{j=1}^{k}u_{j}$ and $y=\sum_{j=1}^{k}\hat\alpha_{t_{j}}(u_{j})$, set $v_{j}=\hat\alpha_{t_{j}}(u_{j})$ and $s_{j}=t_{j}^{-1}$ for $j=1,\dots k$. It clearly follows that

\[\sum_{j=1}^{k}v_{j}=y\quad\mbox{and}\quad \sum_{j=1}^{k}\hat\alpha_{s_{j}}(v_{j})=\sum_{j=1}^{k}\hat\alpha_{t_{j}^{-1}}(\hat\alpha_{t_{j}}(u_{j}))=\sum_{j=1}^{k}u_{j}=x\]
whence $y\sim_{\alpha} x$. Transitivity is a little harder, and here is where the fact that $K_{0}(A)^{+}$ has the Riesz refinement property will surface. To that end, suppose $x\sim_{\alpha} y\sim_{\alpha} z$ via
\[x=\sum_{j=1}^{k}u_{j},\quad y=\sum_{j=1}^{k}\hat\alpha_{t_{j}}(u_{j})\quad\mbox{and}\quad  y=\sum_{j=1}^{l}v_{j},\quad z=\sum_{j=1}^{l}\hat\alpha_{s_{j}}(v_{j}).\]
Since $\sum_{j=1}^{k}\hat\alpha_{t_{j}}(u_{j})=\sum_{j=1}^{l}v_{j}$ and $K_{0}(A)$ has the interpolation properties, there are elements $\{w_{ij} : 1\leq j\leq l, 1\leq i\leq k\}\subset K_{0}(A)^{+}$ such that
\[\sum_{j=1}^{l}w_{ij}=\hat\alpha_{t_{i}}(u_{i})\quad\mbox{and}\quad \sum_{i=1}^{k}w_{ij}=v_{j}.\]
We then compute
\[\sum_{i,j}\hat\alpha_{s_{j}t_{i}}(\hat\alpha_{t_{i}^{-1}}(w_{ij}))=\sum_{i,j}\hat\alpha_{s_{j}}(w_{ij})=\sum_{j}\hat\alpha_{s_{j}}\big(\sum_{i}w_{ij}\big)=\sum_{j}\hat\alpha_{s_{j}}(v_{j})=z,\]
while
\[\sum_{i,j}\hat\alpha_{t_{i}^{-1}}(w_{ij})=\sum_{i}\hat\alpha_{t_{i}^{-1}}\big(\sum_{j}w_{ij}\big)=
\sum_{i}\hat\alpha_{t_{i}^{-1}}(\hat\alpha_{t_{i}}(u_{i}))=\sum_{i}u_{i}=x. \]
which gives the desired decomposition for $x\sim_{\alpha} z$.
\end{proof}

We can now make the following definition.

\begin{defn} Let $A$ be a \cstar-algebra such that $K_0(A)^+$ has the Riesz refinement property. Let $\Gamma\rightarrow\Aut(A)$ be an action. We set $S(A,\Gamma,\alpha):=K_{0}(A)^{+}/\sim_{\alpha}$,
and write $[x]_{\alpha}$ for the equivalence class with representative $x\in K_0(A)^+$.
\end{defn}

For a general group action $G\curvearrowright X$ on an arbitrary set, it is not difficult to see that we may define addition on equidecomposability classes. Indeed if $E,F,H,K\subset X$ with $E\cap H=\emptyset$, $F\cap K=\emptyset$, $E\sim F$ and $H\sim K$ then it is routine to verify that $(E\sqcup H)\sim (F\sqcup K)$. This gives an idea for a well defined additive structure on $S(A,\Gamma,\alpha)$. Define addition on classes simply by $[x]_\alpha+[y]_\alpha:=[x+y]_\alpha$ for $x,y$ in $K_{0}(A)^{+}$. It is routine to check that this operation is well defined; indeed if $z\sim_{\alpha} x$ via $x=\sum_{j=1}^{k}u_{j}$ and $z=\sum_{j=1}^{k}\hat\alpha_{t_{j}}(u_{j})$,then
\[ [z]_\alpha+[y]_\alpha=[z+y]_\alpha=\bigg[\sum_{j=1}^{k}\hat\alpha_{t_{j}}(u_{j})+y\bigg]_{\alpha}
=\bigg[\sum_{j=1}^{k}u_{j}+y\bigg]_{\alpha}=[x+y]_{\alpha}=[x]_{\alpha}+[y]_{\alpha}.\]

We make a few elementary observations concerning $S(A,\Gamma,\alpha)$ when $A$ is stably finite. Firstly, $S(A,\Gamma,\alpha)$  is not just a semigroup but an abelian monoid as $[0]_{\alpha}$ is clearly the neutral additive element. Impose the algebraic ordering on $S(A,\Gamma,\alpha)$, that is, set $[x]_{\alpha}\leq [y]_{\alpha}$ if there is a $z\in K_{0}(A)^{+}$ with $[x]_{\alpha}+[z]_{\alpha}=[y]_{\alpha}$. This gives $S(A,\Gamma,\alpha)$ the structure of an abelian preordered monoid. Notice at once that if $x,y\in K_{0}(A)^{+}$ with $x\leq y$ (in the ordering of $K_{0}(A)$) then $[x]_\alpha\leq[y]_{\alpha}$ in $S(A,\Gamma,\alpha)$. To see this, $x\leq y$ implies $y-x:=z\in K_{0}(A)^{+}$, so $[y]_\alpha=[x+z]_\alpha=[x]_\alpha+[z]_\alpha$ which gives $[x]_\alpha\leq[y]_\alpha$. Next, we observe that if $[x]_\alpha=[0]_\alpha$, for some $x$ in $K_{0}(A)^{+}$, then in fact $x=0$. Indeed, say $x=\sum_{i}u_{i}$, and $\sum_{i}\hat\alpha_{t_{i}}(u_{i})=0$ for some elements $t_{i}\in\Gamma$ and $u_{i}\in K_{0}(A)^{+}$, then for each $i$, $\hat\alpha_{t_{i}}(u_{i})=0$ and so $u_{i}=0$ which gives $x=0$. Here we used the important fact that for stably finite algebras $A$, $K_{0}(A)^{+}\cap (-K_{0}(A)^{+})=(0)$. All together, there is an order preserving, faithful, monoid homomorphism
\[\rho: K_{0}(A)^{+}\rightarrow S(A,\Gamma,\alpha)\quad\mbox{given by}\quad \rho(g)=[g]_\alpha.\]

This next fact shows that we have in fact constructed a noncommutative analogue of the type semigroup construction studied in~\cite{Wa}.

\begin{prop} Let $X$ be the Cantor set, $\Gamma$ a discrete group, and $\Gamma\curvearrowright X$ a continuous action with corresponding action $\alpha:\Gamma\rightarrow\Aut(C(X))$. Then the type semigroup $S(X,\Gamma,\mathcal{C})$ is isomorphic to $S(C(X),\Gamma,\alpha)$ constructed above.
\end{prop}

\begin{proof} Let $f\in K_{0}(C(X))^+=C(X;\mathbb{Z})^+$, then we can write $f=\sum_{j=1}^{n}\mathbbm{1}_{E_{j}}$ where the $E_{j}$ are clopen subsets of $X$. Note that such a representation is not unique.

\noindent \textbf{Claim.} Suppose $f=\sum_{j=1}^{n}\mathbbm{1}_{E_{j}}=\sum_{j=1}^{m}\mathbbm{1}_{F_{j}}$, then
\[\bigsqcup_{j=1}^{n}E_{j}\times\{j\}:=E\sim F:=\bigsqcup_{j=1}^{m}F_{j}\times\{j\},\]
so that $[E]=[F]$ in the type semigroup $S(X,\Gamma,\mathcal{C})$.

It is clear that $\cup_{j=1}^{n}E_j=\cup_{j=1}^{m}F_j$. By choosing a common clopen refinement, we may assume that there are disjoint clopen sets $H_1,\dots,H_r$, where $r\geq n,m$, such that each $E_{j}$ and each $F_{j}$ is a union of distinct $H_i$. For each $i=1,\dots,r$ set the multiplicities of the $H_i$ as
\[n_{i}:=\big|\{j : H_{i}\subset E_j\}\big|=\big|\{j : H_{i}\subset F_j\}\big|.\]
In this case we have $f=\sum_{i=1}^{r}n_{i}\mathbbm{1}_{H_i}$. For each pair $(i,j)$ set
$$
\Delta_{i,j}=
\begin{cases}
H_i, & \text{ if } H_{i}\subset E_j\\
\emptyset & \text{ if } H_{i}\cap E_{j}=\emptyset
\end{cases}
$$
With a $j$ fixed we run through all the $H_i$ and get $\bigsqcup_{i=1}^{r}\Delta_{i,j}\times\{j\}=E_j\times\{j\}$. Then
\[E=\bigsqcup_{j=1}^{n}E_j\times\{j\}=\bigsqcup_{j=1}^{n}\bigsqcup_{i=1}^{r}\Delta_{i,j}\times\{j\}
=\bigsqcup_{i=1}^{r}\bigsqcup_{j=1}^{n}\Delta_{i,j}\times\{j\}\sim\bigsqcup_{i=1}^{r}\bigsqcup_{j=1}^{n_i}H_i\times\{j\}:=H.\]
By a similar argument $F\sim H$, and transitivity gives $E\sim F$ and the Claim is thus proved.

We now define a map $\psi:K_{0}(C(X))^+\rightarrow S(X,\Gamma,\mathcal{C})$ by
\[\psi(f)=\bigg[\bigsqcup_{j=1}^{n}E_j\times\{j\}\bigg]\]
where $f$ has representation $f=\sum_{j=1}^{n}\mathbbm{1}_{E_j}$ with $E_j\subset X$ clopen. Thanks to the Claim, this map is well defined as any representation of $f$ will do. Also, it is routine to check that $\psi$ is additive and onto. Moreover, $\psi$ is invariant under the equivalence $\sim_\alpha$. To see this, suppose $f,g\in K_{0}(C(X))^+$ and $f\sim_{\alpha}g$. By definition and by writing members of $K_{0}(C(X))^+$ as sums of indicator functions on clopen sets we can find clopen sets $E_1,\dots,E_n\in\mathcal{C}$ and group elements $t_1,\dots,t_n\in\Gamma$ with
\[f=\sum_{j=1}^{n}\mathbbm{1}_{E_j},\quad \text{and }\quad g=\sum_{j=1}^{n}\mathbbm{1}_{t_j.E_j}.\]
Since $\bigsqcup_{j=1}^{n}E_j\times\{j\}\sim\bigsqcup_{j=1}^{n}t_j.E_j\times\{j\}$ we get that $\psi(f)=\psi(g)$. The map $\psi$ thus descends to a surjective monoid homomorphism $\overline{\psi}:S(C(X),\Gamma,\alpha)\rightarrow S(X,\Gamma,\mathcal{C})$ with $\overline{\psi}([f]_\alpha)=\psi(f)$. To establish injectivity we construct a left inverse $\varphi:S(X,\Gamma,\mathcal{C})\rightarrow S(C(X),\Gamma,\alpha)$ as follows. Set
\[\varphi\bigg(\bigg[\bigsqcup_{j=1}^{n}E_j\times\{j\}\bigg]\bigg)=\bigg[\sum_{j=1}^{n}\mathbbm{1}_{E_j}\bigg]_{\alpha}.\]
To show that $\varphi$ is well defined, suppose $E=\bigsqcup_{j=1}^{n}E_j\times\{j\}\sim F=\bigsqcup_{j=1}^{m}F_j\times\{j\}$, then there exist $l\in\mathbb{N}$, $C_k\in\mathcal{C}$, $t_k\in\Gamma$ and natural numbers $n_k,m_k$ for $k=1,\dots,l$, such that
\[E=\bigsqcup_{k=1}^{l}C_k\times\{n_k\},\qquad F=\bigsqcup_{k=1}^{l}t_k.C_k\times\{m_k\}.\]
For each fixed $j$, we see that $\bigsqcup_{\{k:\ n_k=j\}}C_k=E_j$, so $\sum_{\{k :\ n_k=j\}}\mathbbm{1}_{C_k}=\mathbbm{1}_{E_j}$. Therefore
\[\sum_{j=1}^{n}\mathbbm{1}_{E_j}=\sum_{j=1}^{n}\sum_{\{k:\ n_k=j\}}\mathbbm{1}_{C_k}=\sum_{k=1}^{l}\mathbbm{1}_{C_k}\sim_{\alpha}\sum_{k=1}^{l}\mathbbm{1}_{t_k.C_k}=\sum_{j=1}^{n}\mathbbm{1}_{F_j}.\]
where the last equality follows from same reasoning. It follows that $\varphi([E])=\varphi([F])$. Also $\varphi$ is clearly additive and onto. For an element $[f]_{\alpha}\in S(C(X),\Gamma,\alpha)$, where $f$ has representation $f=\sum_{j=1}^{n}\mathbbm{1}_{E_{j}}$, we see that
\[\varphi\circ\overline{\psi}([f]_{\alpha})=\varphi\circ\psi(f)=\varphi\bigg(\bigg[\bigsqcup_{j=1}^{n}E_j\times\{j\}\bigg]\bigg)
=\bigg[\sum_{j=1}^{n}\mathbbm{1}_{E_j}\bigg]_{\alpha}=[f]_{\alpha}.\]
We conclude that $\overline{\psi}$ is a monoid isomorphism. Since both monoids are preordered with the algebraic ordering $\overline{\psi}$ is actually an isomorphism of preordered monoids.
\end{proof}

Next we look at how $(\Gamma,k,l)$-paradoxically  is reflected in our monoid $S(A,\Gamma,\alpha)$.

\begin{lem}\label{klpar} Let $A$ be a stably finite C*-algebra such that $K_{0}(A)^{+}$ has Riesz refinement, and let $\alpha:\Gamma\rightarrow\Aut(A)$ be an action. Then an element $0\neq x\in K_{0}(A)^{+}$ is $(\Gamma,k,l)$-paradoxical if and only if $k[x]\leq l[x]$ in $S(A,\Gamma,\alpha)$.
\end{lem}

\begin{proof} Suppose $0\neq x\in K_{0}(A)^{+}$ is $(\Gamma,k,l)$-paradoxical. Then $kx\leq\sum_{j=1}^{n}x_{j}$ and $\sum_{j=1}^{n}\hat\alpha_{t_{j}}(x_{j})\leq lx$ for some $x_{j}$ in $K_{0}(A)^{+}$ and $t_{j}$ in $\Gamma$. Then from our above remarks:
\[k[x]_\alpha=[kx]_\alpha\leq \bigg[\sum_{j=1}^{n}x_{j}\bigg]_\alpha=\bigg[\sum_{j=1}^{n}\hat\alpha_{t_{j}}(x_{j})\bigg]_\alpha\leq [lx]_\alpha=l[x]_\alpha.\]

Now assume $k[x]_\alpha\leq l[x]_\alpha$ for integers $k>l>0$. Then for some $z$ in $K_{0}(A)^{+}$ we have \[[kx+z]_\alpha=[kx]_\alpha+[z]_=k[x]_\alpha+[z]_\alpha=l[x]_\alpha=[lx]_\alpha\].
By definition there are elements $x_{1},\dots,x_{n}$ in $K_{0}(A)^{+}$ and $t_{1},\dots,t_{n}\in\Gamma$ with
\[kx\leq kx+z=\sum_{j=1}^{l}x_{j}\quad\mbox{and}\quad \sum_{j=1}^{l}\hat\alpha_{t_{j}}(x_{j})=lx,\]
which witnesses the $(\Gamma,k,l)$-paradoxicality of $x$. The proof is complete.
\end{proof}

Before going any further let us recall some terminology. Let $(W,\leq)$ be a preordered abelian monoid. For positive integers $k>l>0$, we say that an element $\theta\in W$ is \emph{$(k,l)$-paradoxical} provided that $k\theta\leq l\theta$. If $\theta$ fails to be paradoxical for all pairs of integers $k>l>0$, call $\theta$ \emph{completely non-paradoxical}. Note that $\theta$ is completely non-paradoxical if and only if $(n+1)\theta\nleq n\theta$ for all $n\in\mathbb{N}$. The above lemma basically states that in its setting, an element $x\in K_{0}(A)^{+}$ is completely non-paradoxical with respect to the action $\hat\alpha$ exactly when $[x]_\alpha$ is completely non-paradoxical in the preordered abelian monoid $S(A,\Gamma,\alpha)$. An element $\theta$ in $W$ is said to \emph{properly infinite} if $2\theta\leq\theta$, that is, if it is $(2,1)$-paradoxical. If every member of $W$ is properly infinite then $W$ is said to be \emph{purely infinite}. A \emph{state} on $W$ is a map $\nu:W\rightarrow[0,\infty]$ which is additive, respects the preordering $\leq$, and satisfies $\nu(0)=0$. If a state $\beta$ assumes a value other than $0$ or $\infty$, $\beta$ it said to be \emph{non-trivial}.  The monoid $W$ is said to be \emph{almost unperforated} if, whenever $\theta,\eta\in W$, and $n,m\in\mathbb{N}$ are such that $n\theta\leq m\eta$ and $n>m$, then $\theta\leq\eta$.

The following is a deep result of Tarski, which will be the main tool in establishing a converse to Proposition~\ref{Para1}. A proof can be found in~\cite{Wa}.

\begin{thm}\label{Tarski} Let $(W,+)$ be an abelian monoid equipped with the algebraic ordering, and let $\theta$ be an element of $W$. Then the following are equivalent:
\begin{enumerate}
\item $(n+1)\theta\nleq n\theta$ for all $n\in\mathbb{N}$, that is $\theta$ is completely non-paradoxical.
\item There is a non-trivial state $\nu: W\rightarrow[0,\infty]$ with $\nu(\theta)=1$.
\end{enumerate}
\end{thm}

We mean to apply Theorem~\ref{Tarski} to our preordered monoid $S(A,\Gamma,\alpha)$. Note that such a $\nu$, which arises in the landscape of complete non-paradoxicality will not in general be finite on all of $S(A,\Gamma,\alpha)$. One needs the right condition on the action $\alpha$, or more precisely, $\hat\alpha$,  to guarantee finiteness everywhere. Suppose we considered $\theta=[u]_\alpha$ as in Theorem~\ref{Tarski}, where $u=[1]_0$ is the order unit in $K_{0}(A)$. If we compose the state $\nu$ with the the above $\rho:K_{0}(A)^{+}\rightarrow S(A,\Gamma,\alpha)$, this would give us, in a sense, an invariant `state' at the $K$-theoretic level, but perhaps not finitely valued everywhere, but with a finite value at $[1]_0$. To ensure finiteness at every $x\in K_{0}(A)^{+}$ we would require that finitely many $\Gamma$-iterates of $x$ lie above $[1]_{0}$. This is exactly the notion of $K$-theoretic minimality we looked at in Section 3.1.

\begin{prop}\label{CNP iff state} Let $A$ be a stably finite unital C*-algebra for which $K_{0}(A)^{+}$ has Riesz refinement ($\sr(A)=1$ and $\RR(A)=0$ for example). Let  $\alpha:\Gamma\rightarrow\Aut(A)$ be an action on $A$. Consider the following properties.
\begin{enumerate}
\item For every $0\neq g\in K_{0}(A)^{+}$, there is a faithful $\Gamma$-invariant positive group homomorphism $\beta:K_{0}(A)\rightarrow \mathbb{R}$ with $\beta(g)=1$, $(\Gamma$-invariant in the sense that $\beta\circ\hat\alpha=\beta$ on $K_{0}(A)).$
\item There is a faithful $\Gamma$-invariant state $\beta$ on $(K_{0}(A),K_{0}(A)^{+},[1]_{0})$.
\item $\alpha$ is completely non-paradoxical.
\end{enumerate}
Then we have $(1)\Rightarrow(2)\Rightarrow(3)$. If the action $\alpha$ is minimal, then $(3)\Rightarrow(1)$ whence all the conditions are equivalent.
\end{prop}

\begin{proof}$(1)\Rightarrow(2)$: Simply take $g=[1]_{0}$.

$(2)\Rightarrow(3)$: Assume that $x\in K_{0}(A)^{+}$ is $(\Gamma,k,l)$-paradoxical for some integers $k>l>0$ with paradoxical decomposition $\sum_{j}^{n}x_{j}\geq kx$ and $\sum_{j}^{n}\hat\alpha_{t_{j}}(x_{j})\leq lx$ for certain $x_{j}\in K_{0}(A)^{+}$ and $t_{j}\in\Gamma$. Apply the $\hat\alpha$-invariant state $\beta$ and get
\[k\beta(x)=\beta(kx)\leq\beta\big(\sum_{j}^{n}x_{j}\big)=\sum_{j}^{n}\beta(x_{j})=\sum_{j}^{n}\beta(\hat\alpha_{t_{j}}(x_{j}))\\
=\beta\big(\sum_{j}^{n}\hat\alpha_{t_{j}}(x_{j})\big)\leq\beta(lx)=l\beta(x).\]

Now since $\beta$ is faithful, we may divide by $\beta(x)>0$ and get $k\leq l$ which is absurd.

Assuming the action $\alpha$ is minimal we prove $(3)\Rightarrow(1)$. Fix a non-zero $g\in K_{0}(A)^{+}$. Since the action is completely non-paradoxical, it follows from Lemma~\ref{klpar} that for every positive integer $n$, $(n+1)[g]_{\alpha}\nleq n[g]_{\alpha}$. Tarski's Theorem then states that $S(A,\Gamma,\alpha)$ admits a non-trivial state $\nu:S(A,\Gamma,\alpha)\rightarrow[0,\infty]$ with $\nu([g]_{\alpha})=1$.\\

\noindent\textbf{Claim:} $\nu$ is finite.\\

To see this, employ $K$-minimality of the action to obtain group elements $t_{1},\dots,t_{n}$ such that $\sum_{j=1}^{n}\hat\alpha_{t_{j}}(g)\geq [1]_{0}$. Now for an arbitrary $[x]_\alpha$ in $S(A,\Gamma,\alpha)$ with $x$ belonging to $K_{0}(A)^{+}$, there is a positive integer $m$ with $x\leq m[1]_{0}\leq m\sum_{j=1}^{n}\hat\alpha_{t_{j}}(g)$. Therefore
\[ [x]_\alpha\leq\bigg[m\sum_{j=1}^{n}\hat\alpha_{t_{j}}(g)\bigg]_\alpha=m\bigg[\sum_{j=1}^{n}\hat\alpha_{t_{j}}(g)\bigg]_\alpha=m[ng]_\alpha=mn[g]_\alpha.\]
Applying $\nu$ yields $\nu([x]_\alpha)\leq\nu(mn[g]_\alpha)=mn\nu([g]_\alpha)=mn$. The Claim is therefore proved.

We now compose $\nu$ with our above $\rho:K_{0}(A)^{+}\rightarrow S(A,\Gamma,\alpha)$ to yield $\beta':K_{0}(A)^{+}\rightarrow([0,\infty),+)$ a finite order preserving monoid homomorphism given by $\beta'(x)=\nu([x]_\alpha)$. Note how $\beta'$ is invariant under the action $\hat\alpha:\Gamma\curvearrowright K_{0}(A)^{+}$. Indeed, for $t$ in $\Gamma$, and $x$ in $K_{0}(A)^{+}$,
\[\beta'(\hat\alpha_{t}(x))=\nu([\hat\alpha_{t}(x)]_\alpha)=\nu([x]_\alpha)=\beta'(x).\]
By universality of the Grothendieck enveloping group construction, there is a unique extension of $\beta'$ to a group homomorphism on all of $K_{0}(A)$, which we will denote as $\beta$, given simply by $\beta(x-y)=\beta'(x)-\beta'(y)$ for $x,y$ in $K_{0}(A)^{+}$. Clearly $\beta$ is still $\Gamma$-invariant. The final product is a bona fide $\Gamma$-invariant positive group homomorphism $\beta:K_{0}(A)\rightarrow\mathbb{R}$, with $\beta(g)=1$. We now show how $\beta$ is faithful which will complete this direction. Assume $0\neq x\in K_{0}(A)^{+}$. Minimality ensures the existence of group elements $t_{1},\dots t_{n}$ with $\sum_{j=1}^{n}\hat\alpha_{t_{j}}(x)\geq [1]_{0}$. Now we find a positive integer $m$ for which $m[1]_{0}\geq g$, so that $m\big(\sum_{j=1}^{n}\hat\alpha_{t_{j}}(x)\big)\geq g$. Applying $\beta$ gives
\[1=\beta(g))\leq \beta\big(m\big(\sum_{j=1}^{n}\hat\alpha_{t_{j}}(x)\big)\big)=m\big(\sum_{j=1}^{n}\beta(\hat\alpha_{t_{j}}(x))\big)
=m\big(\sum_{j=1}^{n}\beta(x)\big)=mn\beta(x)\]
thus $\beta(x)\neq 0$ and $\beta$ is indeed faithful.
\end{proof}

We now are ready to establish the long desired converse.

\begin{thm}\label{stablyfinitecross} Let $A$ be a stably finite unital C*-algebra for which $K_{0}(A)^{+}$ has Riesz refinement ($\sr(A)=1$ and $\RR(A)=0$ for example). Let  $\alpha:\Gamma\rightarrow\Aut(A)$ be a minimal action on $A$. Consider the following properties.
\begin{enumerate}
\item There is an $\Gamma$-invariant faithful tracial state $\tau:A\rightarrow\mathbb{C}$.
\item $A\rtimes_{\lambda}\Gamma$ admits a faithful tracial state.
\item $A\rtimes_{\lambda}\Gamma$ is stably finite.
\item $\alpha$ is completely non-paradoxical.
\item There is a faithful $\Gamma$-invariant state $\beta$ on $(K_{0}(A),K_{0}(A)^{+},[1]_{0})$.
\end{enumerate}

Then we have the following implications:
\[(1)\Leftrightarrow(2)\Rightarrow(3)\Rightarrow(4)\Rightarrow(5).\]

If $A$ is exact and projections are total in $A$ (e.g. $RR(A)=0$) then $(5)\Leftrightarrow(1)$. Furthermore, if $A$ is AF and $\Gamma$ is a free group, then $(1)$ through $(5)$ are all equivalent to $A\rtimes_{\lambda}\Gamma$ being MF.
\end{thm}

\begin{proof} It is well known that $(1)\Leftrightarrow(2)\Rightarrow(3)$. Also, $(3)\Rightarrow(4)$ is Proposition~\ref{SF implies CNP} and $(4)\Rightarrow(5)$ is Proposition~\ref{CNP iff state}.


$(5)\Rightarrow(1)$: Since $A$ exact, such a $\beta$ arises from a tracial state $\tau:A\rightarrow\mathbb{C}$, via $\tau(p)=\beta([p])$ for any projection $p\in A$. We need only to show the $\Gamma$-invariance of $\tau$. For any $s\in\Gamma$ and projection $p$ in $A$,
\[\tau(\alpha_{s}(p))=\beta([\alpha_{s}(p)])=\beta\circ\hat\alpha_{s}([p])=\beta([p])=\tau(p).\]
Using linearity, continuity, and the fact that the projections are total in $A$, it follows that $\tau(\alpha_{s}(a))=\tau(a)$ for every $a\in A$ and $s\in\Gamma$ which yields the invariance.

Now we let $\Gamma=\mathbb{F}_{r}$ and $A$ an AF algebra. In~\cite{Ra} the author shows that $A\rtimes_\lambda \mathbb{F}_{r}$ is MF if and only if it is stably finite.

\end{proof}

Recall that if a discrete group $\Gamma$ is amenable, and $K$ is a Frechet space, then any continuous action $\Gamma\curvearrowright K$ admits a fixed point.

\begin{cor} Let $A$ be a simple, unital, AF algebra and $\Gamma$ a discrete amenable group. Then any action $\alpha:\Gamma\rightarrow\Aut(A)$ is completely non-paradoxical.
\end{cor}

\begin{proof} The group $\Gamma$ acts on the tracial state space $T(A)$ and thus has a fixed point. Now refer to the previous Theorem.
\end{proof}

\subsection{Purely Infinite Crossed Products}

A continuous action $\Gamma\curvearrowright X$ of a discrete group on a compact Hausdorff space is called a \emph{strong boundary action} if $X$ has at least three points and for every pair $U,V$ of non-empty open subsets of $X$ there exists $t\in\Gamma$ with $t.U^c\subset V$. Laca and Spielberg showed in~\cite{LS} that if $\Gamma\curvearrowright X$ is a strong boundary action and the induced action $\Gamma\curvearrowright C(X)$ is properly outer then $C(X)\rtimes_{\lambda}\Gamma$ is purely infinite and simple.

Jolissaint and Robertson~\cite{JR} made a generalization valid in the noncommutative setting. They termed an action $\alpha:\Gamma\rightarrow\Aut(A)$ as  \emph{$n$-filling} if, for all $a_1,\dots,a_n\in A^+$, with $\|a_{j}\|=1$, $1\leq j\leq n$, and for all $\varepsilon>0$, there exist $t_1,\dots,t_n\in\Gamma$ such that $\sum_{j=1}^{n}\alpha_{t_{j}}(a_j)\geq (1-\varepsilon)1_A$. They showed that $A\rtimes_{\lambda}\Gamma$ is purely infinite and simple provided that the action is properly outer and $n$-filling and every corner $pAp$ of $A$ is infinite dimensional. Using ordered $K$-theoretic dynamics we shall provide an alternate simpler proof of this result below, albeit for a smaller class of algebras.

The following lemma contains ideas from Lemma 3.1 of~\cite{SR}.

\begin{lem}\label{CuntzSmaller} Let $(A,\Gamma,\alpha)$ be a C*-dynamical system with $A$ separable and $\Gamma$ countable and discrete. Assume that $\alpha$ is properly outer. Then for every non-zero $b\in (A\rtimes_{\lambda}\Gamma)^+$ there is a non-zero $a\in A^+$ with $a\precsim b$.
\end{lem}

\begin{proof} We know that $\mathbb{E}(b)\neq0$ since $b$ is non-zero and $\mathbb{E}$ is faithful. Set $b_{1}=b/\|\mathbb{E}(b)\|$ so that $\|\mathbb{E}(b_{1})\|=1$. Let $0<\varepsilon<1/16$. Find a $\delta>0$ with $\frac{\delta(1+\|b_1\|)}{1-\delta}<\varepsilon$. Next find a non-zero positive $c\in C_{c}(\Gamma,A)^+$ with $\|c-b_{1}\|<\delta$. Write $c=\sum_{s\in F}c_{s}u_{s}$ where $F$ is a finite subset of $\Gamma$. Note that $\mathbb{E}(c)=c_{e}\neq0$, and also $\big|1-\|c_{e}\|\big|\leq\delta$. Setting $d=c/\|c_{e}\|$ we estimate
\begin{align*}\|b_1-d\|=&\frac{1}{\|c_e\|}\big\|\|c_e\|b_1-c\big\|=\frac{1}{\|c_e\|}\big\|\|c_e\|b_1-b_1+b_1-c\big\|
\leq\frac{1}{\|c_e\|}\big(|\|c_e\|-1|\|b_1\|+\|b_1-c\|\big)\\ &\leq\frac{1}{1-\delta}(\delta\|b_1\|+\delta)=\frac{\delta}{1-\delta}(1+\|b_{1}\|)<\varepsilon.
\end{align*}

Now let $\eta>0$ be so small that $|F|\eta<1/8$. Since $A$ is separable and $\alpha$ is properly outer, we apply Lemma 7.1 of~\cite{OP} and obtain an element $x\in A^+$ with $\|x\|=1$ satisfying
\[\|x\mathbb{E}(d)x\|=\|xd_ex\|>\|d_e\|-\eta=1-\eta, \qquad \|xd_s\alpha_s(x)\|<\eta\quad \forall s\in F\setminus\{e\}.\]
Therefore we have
\begin{align*}\|x\mathbb{E}(d)x-xdx\|&\leq\bigg\|\sum_{s\in F\setminus\{e\}}xd_su_sx\bigg\|\leq \sum_{s\in F\setminus\{e\}}\|xd_su_sx\|\\
&=\sum_{s\in F\setminus\{e\}}\|xd_su_sxu_s^*\|=\sum_{s\in F\setminus\{e\}}\|xd_s\alpha_s(x)\|\leq |F|\eta<1/8.
\end{align*}
A straightforward use of the triangle inequality now gives
\[\|x\mathbb{E}(b_1)x-xb_1x\|\leq 2\varepsilon+1/8<1/4,\qquad \|x\mathbb{E}(b_1)x\|\geq3/4.\]
Let $a:=(x\mathbb{E}(b_1)x-1/2)_+$. Then $a\in A$ and $a\neq0$ since $\|x\mathbb{E}(b_1)x\|>1/2$. Also by Proposition 2.2 of~\cite{R} we know $a\precsim xb_1x\precsim b_1\precsim b$.
\end{proof}

Theorem 4.1 in~\cite{SR} concentrates on the commutative case. We, however, make the observation that the same proof holds true for noncommutative algebras. Recall that a \cstar-algebra $A$ has property (SP) if every non-zero hereditary subalgebra admits a non-zero projection.

\begin{thm}\label{purelyinfinite} Let $(A,\Gamma,\alpha)$ be a C*-dynamical system with $A$ separable with property (SP) and $\Gamma$ countable and discrete. Assume that $\alpha$ is minimal and properly outer. Then the following are equivalent:
\begin{enumerate}
\item $A\rtimes_{\lambda}\Gamma$ is purely infinite.
\item Every non-zero projection $p$ in $A$ is properly infinite in $A\rtimes_{\lambda}\Gamma$.
\end{enumerate}
\end{thm}

\begin{proof} $(1)\Rightarrow(2)$: Every non-zero projection in any purely infinite algebra is properly infinite.

$(2)\Rightarrow(1)$: By Theorem 7.2 in~\cite{OP} we know that the reduced crossed product $A\rtimes_{\lambda}\Gamma$ is simple. Therefore, it suffices to show that every hereditary subalgebra admits an infinite projection. To this end, let $B\subset A\rtimes_{\lambda}\Gamma$ be a hereditary \cstar-subalgebra and let $0\neq b\in B$. By lemma~\ref{CuntzSmaller} there is a non-zero $a$ in $A$ with $a\precsim b$. Since $A$ has property (SP), the hereditary subalgebra of $A$ generated by $a$, $H_a=\overline{aAa}$, contains a non-zero projection $q\in H_a$. By our assumption $q$ is properly infinite relative to $A\rtimes_{\lambda}\Gamma$, and $q\precsim a\precsim b$. Since $q$ is a projection, there is a $z\in A\rtimes_{\lambda}\Gamma$ with $q=z^*bz$. Now consider $v:=b^{1/2}z$. Then $q=v^*v\sim vv^*=b^{1/2}zz^*b^{1/2}\in B$. Thus $p:=vv^*$ is the desired properly infinite projection in $B$.
\end{proof}

We now embark on studying to what extent paradoxical systems $(A,\Gamma,\alpha)$  characterize purely infinite reduced crossed product algebras $A\rtimes_{\lambda}\Gamma$.

\begin{prop}~\label{propinf}  Let $(A,\Gamma,\alpha)$ be a C*-system for which $A$ has cancellation and $K_{0}(A)^+$ has the Riesz refinement property. Let $0\neq r\in\mathcal{P}(A)$ and set $g=[r]_{0}\in K_{0}(A)^{+}$. The following properties are equivalent:
\begin{enumerate}
\item There exist $x,y\in C_{c}(\Gamma,A)$ that satisfy $x^*x=r=y^*y$, $xx^*\perp yy^*$, $xx^*\leq r$, $yy^*\leq r$, and whose coefficients are partial isometries.
\item $g$ is $(k,1)$-paradoxical for some $k\geq2$.
\item $\theta=[g]_{\alpha}$ is infinite in $S(A,\Gamma,\alpha)$.
\end{enumerate}
\end{prop}

\begin{proof}$(1)\Rightarrow(2)$: Write $x=\sum_{s\in F}u_{s}v_{s}$ and $y=\sum_{s\in L}u_{s}w_{s}$ where $F,L\subset\Gamma$ are finite subsets, and $v_{s},w_{s}\in A$ are partial isometries. For each $s$ in $F$ set $p_{s}:=v_{s}^*v_{s}$ and $p_{s}':=v_{s}v_{s}^*$. Similarly for every $s\in L$ set $q_{s}:=w_{s}^{*}w_{s}$ and $q_{s}':=w_{s}w_{s}^*$. If we apply the conditional expectation $\mathbb{E}:A\rtimes_{\lambda}\Gamma\rightarrow A$ to the equality $r=x^*x$ we get
\[r=\mathbb{E}(r)=\mathbb{E}\bigg(\sum_{s,t\in F}v_{s}^*u_{s}^*u_{t}v_{t}\bigg)=\sum_{s,t\in F}\mathbb{E}(v_{s}^*u_{s}^*u_{t}v_{t})=\sum_{s\in F}v_{s}^*v_{s}=\sum_{s\in F}p_{s}.\]
The second to last equality follows from the fact that for $s,t\in F$ we have
\[\mathbb{E}(v_{s}^*u_{s}^*u_{t}v_{t})=\mathbb{E}(v_{s}^*u_{s^{-1}t}v_{t}(u_{s^{-1}t})^*u_{s^{-1}t})
=\mathbb{E}(v_{s}^*\alpha_{s^{-1}t}(v_{t})u_{s^{-1}t})=\delta_{s,t}v_{s}^{*}v_{s}.\]
Therefore, the projections $p_{s}$ are mutually orthogonal subprojections of $r$ that sum to $r$. Similarly all the $q_{s}$, for $s\in L$, are mutually orthogonal subprojections of $r$ with $r=\sum_{s\in L}q_{s}$. Thus, in $K_{0}(A)^{+}$ we have
\[\sum_{s\in F}[p_{s}]_{0}+\sum_{s\in L}[q_{s}]_{0}=\bigg[\sum_{s\in F}p_{s}\bigg]_{0}+\bigg[\sum_{s\in F}q_{s}\bigg]_{0}=2[r]_{0}.\]

Now we note that for $s,t$ in $F$ with $s\neq t$ we have $v_{s}v_{t}^*=v_sv_s^*v_sv_t^*v_tv_t^*=v_sp_sp_tv_t^*=0$. Computing $xx^*$ we get
\[xx^*=\sum_{s,t\in F}u_{s}v_{s}v_t^*u_t^*=\sum_{s\in F}u_sv_sv_s^*u_s^*=\sum_{s\in F}\alpha_{s}(p_{s}').\]
Similarly $yy^*=\sum_{s\in L}\alpha_{s}(q_{s}')$. From
\[\sum_{s\in F}\alpha_{s}(p_{s}')+\sum_{s\in L}\alpha_{s}(q_{s}')=xx^*+yy^*\leq r\]
we conclude that the projections $\alpha_{s}(p_{s}'),\alpha_{s}(q_{s}')$ are mutually orthogonal subprojections of $r$ whence in $K_{0}(A)$ we have
\begin{align*}[r]_{0}\geq\bigg[\sum_{s\in F}\alpha_{s}(p_{s}')+\sum_{s\in L}\alpha_{s}(q_{s}')\bigg]_{0}&=\sum_{s\in F}[\alpha_{s}(p_{s}')]_{0}+\sum_{s\in L}[\alpha_{s}(q_{s}')]_0\\&=\sum_{s\in F}\hat\alpha_{s}([p_{s}']_{0})+\sum_{s\in L}\hat\alpha_{s}([q_{s}']_{0})=\sum_{s\in F}\hat\alpha_{s}([p_{s}]_{0})+\sum_{s\in L}\hat\alpha_{s}([q_{s}]_{0}).
\end{align*}
Therefore $g=[r]_{0}$ is (2,1)-paradoxical.

$(2)\Rightarrow(1)$: Suppose $\sum_{j=1}^{n}x_{j}\geq k[r]_0$ and $\sum_{j=1}^{n}\hat\alpha_{t_{j}}(x_{j})\leq[r]_0$ for some $k\geq2$, group elements $t_{1},\dots,t_{n}\in\Gamma$, and $x_{j}\in K_{0}(A)^{+}$. Since $k[r]_0\geq2[r]_0$ we may assume $k=2$. For some $u\in K_{0}(A)^+$ we then have $\sum_{j=1}^{n}x_{j}=[r]_0+[r]_0+u$. Refinement implies that there are subsets $\{y_{j}\}_{j=1}^{n}$, $\{z_{j}\}_{j=1}^{n}$ and $\{u_{j}\}_{j=1}^{n}$ of $K_{0}(A)^{+}$ with
\[\sum_{j=1}^{n}y_{j}=[r],\quad \sum_{j=1}^{n}z_{j}=[r],\quad \sum_{j=1}^{n}u_{j}\geq0,\quad\mbox{and}\quad x_{j}=y_{j}+z_{j}+u_{j},\quad \forall j.\]
Using the fact that $A$ has cancelation we know that there are mutually orthogonal projections $p_{j}\in \mathcal{P}(A)$ with $[p_{j}]_0=y_j$ for $j=1,\dots,n$. Similarly there are mutually orthogonal projections $q_{j}\in\mathcal{P}(A)$ with $[q_j]_0=z_j$ for $j=1,\dots,n$. Therefore,
\begin{align*}
\sum_{j}[\alpha_{t_{j}}(p_{j})]_0+\sum_{j}[\alpha_{t_{j}}(q_{j})]_0&=\sum_{j}\hat\alpha_{t_{j}}(y_j)+\sum_{j}\hat\alpha_{t_{j}}(z_j)
\\&\leq\sum_{j}\hat\alpha_{t_{j}}(y_j)+\sum_{j}\hat\alpha_{t_{j}}(z_j)+\sum_{j}\hat\alpha_{t_{j}}(u_j)=\sum_{j}\hat\alpha_{t_{j}}(x_j)\leq[r]_0.
\end{align*}
We again use the fact that $A$ has cancellation and find mutually orthogonal subprojections of $r$ $e_{1},\dots,e_{n};f_{1},\dots,f_{n}\in\mathcal{P}(A)$ with $[e_{j}]_0=[\alpha_{t_{j}}(p_{j})]_0$ and $[f_{j}]_0=[\alpha_{t_{j}}(q_{j})]_0$ for every $j$. Cancellation also implies that there are partial isometries $v_{j}$ and $w_{j}$ in $A$ with
\[v_j^*v_j=\alpha_{t_{j}}(p_{j}),\quad v_jv_j^*=e_j,\quad w_j^*w_j=\alpha_{t_{j}}(q_{j}),\quad w_jw_j^*=f_j.\]

Now set $a:=\sum_{j=1}^{n}v_{j}u_{j}$ and $b:=\sum_{j=1}^{n}w_{j}u_{j}$ where $u_{j}=u_{t_{j}}$. Note that for $i\neq j$ we compute $v_j^*v_i=v_j^*v_jv_j^*v_iv_i^*v_i=v_j^*e_je_iv_i=0$, so
\[a^*a=\sum_{i,j}u_j^*v_j^*v_iu_i=\sum_{j}u_{j}^*v_j^*v_ju_j=\sum_{j}u_{j}^*\alpha_{t_{j}}(p_{j})u_{t_{j}}
=\sum_{j}\alpha_{t_{j}^{-1}}(\alpha_{t_{j}}(p_{j}))=\sum_{j}p_{j}:=p.\]
In order to compute $aa^*$ we note that for $i\neq j$ we have
\[v_ju_ju_i^*v_i^*=v_jv_j^*v_ju_ju_i^*v_i^*v_iv_i^*=v_j\alpha_{t_{j}}(p_{j})u_ju_i^*\alpha_{t_{i}}(p_{i})v_i^*
=v_ju_jp_{j}u_j^*u_ju_i^*u_ip_{i}u_i^*v_i^*=v_ju_jp_{j}p_{i}u_i^*v_i^*=0,\]
whence
\[aa^*=\sum_{i,j}v_ju_ju_i^*v_i^*=\sum_{j}v_ju_ju_j^*v_j^*=\sum_{j}v_{j}v_{j}^*=\sum_{j}e_{j}:=e.\]
Similarly $b^*b=\sum_{j}q_{j}:=q$, and $bb^*=\sum_{j}f_{j}=f$.

Now define $x:=av$ where $v$ is the partial isometry in $A$ with $v^*v=r$ and $vv^*=p$. Such a $v$ exists because $[p]_0=\big[\sum_{j}p_j\big]_0=\sum_{j}[p_j]_0=\sum_{j}y_j=[r]_0$ and $A$ has cancellation. Similarly define $y:=bw$ where $w\in A$ satisfies $w^*w=r$ and $ww^*=q$. We compute
\[x^*x=v^*a^*av=v^*pv=v^*vv^*v=r^2=r,\]
and
\[y^*y=w^*b^*bw=w^*qw=w^*ww^*w=r^2=r.\]
Moreover, since $a$ and $b$ are partial isometries, and $e\perp f$ we have
\[xx^*yy^*=avv^*a^*bww^*b^*=avv^*a^*aa^*bb^*bww^*b^*=avv^*a^*efbww^*b^*=0.\]
Next we observe that $xx^*$ is a subprojection of $r$; indeed, since $e\leq r$,
\[rxx^*=ravv^*a^*=raa^*avv^*a^*=reavv^*a^*=eavv^*a^*=aa^*avv^*a^*=avv^*a^*=xx^*\]
Similarly $yy^*$ is a subprojection of $r$.

Finally we verify that the coefficients of $x$ and $y$ are partial isometries. Write \[x=av=\sum_{j=1}^{n}v_{j}u_{j}v=\sum_{j=1}^{n}v_{j}\alpha_{t_j}(v)u_j,\]
and compute
\[(v_{j}\alpha_{t_j}(v))^*v_{j}\alpha_{t_j}(v)=\alpha_{t_j}(v^*)v_j^*v_{j}\alpha_{t_j}(v)
=\alpha_{t_j}(v^*)\alpha_{t_j}(p_j)\alpha_{t_j}(v)=\alpha_{t_j}(v^*p_jv),\]
but since $p_j\leq p$ for every $j$, $v^*p_jv$ is a projection: $(v^*p_jv)^2=v^*p_jvv^*p_jv=v^*p_jpp_jv=v^*p_jv$. Therefore $\alpha_{t_j}(v^*p_jv)$ is a projection for each $j$ and so the coefficients of $x$, $v_{j}\alpha_{t_j}(v)$, are partial isometries. An identical argument works for the coefficients of $y$. This completes the implication $(2)\Rightarrow(1)$.

$(2)\Leftrightarrow(3)$: By definition $[g]_{\alpha}$ is infinite in $S(A,\Gamma,\alpha)$ if and only if $2[g]_{\alpha}\leq[g]_{\alpha}$, and by Proposition~\ref{klpar}, we know this occurs if and only if $g$ is $(2,1)$-paradoxical. Clearly $g$ is $(2,1)$-paradoxical if and only if $g$ is $(k,1)$-paradoxical for some $k\geq2$.
\end{proof}

At this point we can supply an alternate proof of Jolissaint and Robertson's result using ordered $K$-theory, but first, two basic lemmas. Recall that a partially ordered group $(G,G^+)$ is said to be \emph{non-atomic} if, for every non-zero $g>0$, there is an $h\in G$ with $0<h<g$. 

\begin{lem} If $A$ is a unital stably finite C*-algebra with property (SP) such that $pAp$ is infinite dimensional for every projection $p\in A$, then $(K_{0}(A),K_{0}(A)^+)$ is non-atomic.
\end{lem}

\begin{proof}
Let $0<g=[q]_0$ belong to $K_{0}(A)^+$ for some non-zero $q\in\mathcal{P}_{n}(A)$. Then clearly there is a non-zero $b\in A^+$ with $b\precsim q$. By property (SP) there is a non-zero projection $p\in \overline{bAb}$. A little work gives $p\precsim b$. By hypothesis the corner $pAp$ is infinite dimensional and thus every masa of $pAp$ is infinite dimensional. Inside such an infinite dimensional masa we can find positive elements $a_1,a_2$ of norm one with $a_1a_2=0$. Now find non-zero projections $p_i\in \overline{a_iAa_i}$ for $i=1,2$. Then $p_1,p_2$ are non-zero orthogonal subprojections of $p$. It follows that $g>[p_1]_0>0$.
\end{proof}

\begin{lem} Let $A$ be a unital C*-algebra and $\alpha:\Gamma\rightarrow\Aut(A)$ an action. Consider the following properties:
\begin{enumerate}
\item The action $\alpha$ is $n$-filling.
\item The action $\alpha$ is $W$-$n$-minimal.
\item The action $\alpha$ is $W$-$n$-filling.
\item The action $\alpha$ is $K_0$-$n$-filling..
\end{enumerate}
Then $(1)\Rightarrow(2)\Rightarrow(3)\Rightarrow(4)$.
\end{lem}

\begin{proof}
$(1)\Rightarrow(2)$: Let $x\in W(A)$. We can find a positive norm-one element $b\in A^+$ with $b\precsim x$. By hypothesis there are group elements $t_1,\dots,t_n$ with $\sum_{j=1}^{n}\alpha_{t_{j}}(b)\geq (1/2)1_{A}$. The result follows since
\[\langle1_A\rangle=\langle(1/2)1_{A}\rangle\leq\big\langle\sum_{j=1}^{n}\alpha_{t_{j}}(b)\big\rangle\leq\langle\oplus_{j}\alpha_{t_{j}}(b)\rangle=\sum_{j=1}^{n}\langle\alpha_{t_{j}}(b)\rangle=\sum_{j=1}^{n}\hat\alpha_{t_{j}}(\langle b\rangle)\leq\sum_{j=1}^{n}\hat\alpha_{t_{j}}(x).\]

$(2)\Leftrightarrow(3)$: This was shown in Proposition~\ref{nfilling=nminimal} above.

$(3)\Rightarrow(4)$: This follows from the fact that if $p,q\in\mathcal{P}_{\infty}(A)$ and $\langle p\rangle\leq\langle q\rangle$ in $W(A)$, then $[p]_{0}\leq[q]_{0}$ in $K_{0}(A)$.
\end{proof}

\begin{prop}\label{JRthm} Let $A$ be a C*-algebra with cancellation, property (SP), and for which  $(K_{0}(A), K_{0}(A)^+)$ is non-atomic and $K_{0}(A)^+$ has Riesz refinement (an algebra of real rank zero and stable rank one will do). Let $\alpha:\Gamma\rightarrow\Aut(A)$ be a properly outer action which is $K_{0}$-$n$-filling for some $n\in\mathbb{N}$. Then $A\rtimes_{\lambda}\Gamma$ is simple and purely infinite.
\end{prop}

\begin{proof}
By theorem~\ref{purelyinfinite} it suffices to prove that every projection $p$ in $A$ is properly infinite in $A\rtimes_{\lambda}\Gamma$. Now by Proposition~\ref{propinf} we need only show that $g=[p]_{0}$ in $K_{0}(A)^+$ is  $(2,1)$-paradoxical. Since $K_{0}(A)^+$ is non-atomic we may find non-zero elements $x_1,\dots,x_{2n}\in K_{0}(A)^+$ with $\sum_{j=1}^{2n}x_j\leq g$. By the n-filling property there are group elements $t_{1},\dots,t_{2n}$ with
\[\sum_{j=1}^{n}\hat\alpha_{t_{j}}(x_j)\geq [1]_{0},\quad\mbox{and}\quad \sum_{j=n+1}^{2n}\hat\alpha_{t_{j}}(x_j)\geq [1]_{0}.\]
Together $\sum_{j=1}^{2n}x_j\leq g$ and  $\sum_{j=1}^{2n}\hat\alpha_{t_{j}}(x_j)\geq2[1]_{0}\geq2g$ and thus $g$ is  $(2,1)$-paradoxical.
\end{proof}

The following result  generalizes  Theorem 5.4 of~\cite{SR} to the noncommutative case.

\begin{thm}\label{purelyinfinitecross} Let $A$ be a unital, separable, exact C*-algebra whose projections are total. Moreover, suppose $A$ has cancellation and  $K_{0}(A)^+$ has the Riesz refinement property. Let $\alpha:\Gamma\rightarrow\Aut(A)$ be a minimal and properly outer action. Consider the following properties:
\begin{enumerate}
\item The semigroup $S(A,\Gamma,\alpha)$ is purely infinite.
\item Every non-zero element in $K_{0}(A)^+$ is $(k,1)$-paradoxical for some $k\geq2$.
\item The \cstar-algebra $A\rtimes_{\lambda}\Gamma$ is purely infinite.
\item The \cstar-algebra $A\rtimes_{\lambda}\Gamma$ is traceless.
\item The semigroup $S(A,\Gamma,\alpha)$ admits no non-trivial state.
\end{enumerate}
Then the following implications always hold: $(1)\Leftrightarrow(2)\Rightarrow(3)\Rightarrow(4)\Rightarrow(5)$. If the semigroup $S(A,\Gamma,\alpha)$ is almost unperforated then $(5)\Rightarrow(1)$ and all properties are equivalent.
\end{thm}

\begin{proof}$(1)\Leftrightarrow(2)$: We have already seen that $x\in K_{0}(A)^+$ is $(k,1)$-paradoxical for some $k\geq2$ if and only if $\theta=[x]_{\alpha}$ is infinite in $S(A,\Gamma,\alpha)$.

$(2)\Rightarrow(3)$: Let $r$ be a non-zero projection in $A$. By assumption $[r]_0$ is $(2,1)$-paradoxical, so by lemma~\ref{propinf} $r$ is properly infinite in $A\rtimes_{\lambda}\Gamma$. Then $A\rtimes_{\lambda}\Gamma$ is purely infinite by Theorem~\ref{purelyinfinite}.

$(3)\Rightarrow(4)$: Purely infinite \cstar-algebras are always traceless.

$(4)\Rightarrow(5)$: Suppose $\nu:S(A,\Gamma,\alpha)\rightarrow[0,\infty]$ is a non-trivial state. Suppose $0<\nu([x]_\alpha)<\infty$ where $x\in K_0(A)^+$ is non-zero. Composing with the quotient  map $\rho:K_0(A)^+\rightarrow S(A,\Gamma,\alpha)$ we get an order preserving monoid homomorphism $\beta'=\nu\circ\rho:K_{0}(A)^+\rightarrow[0,\infty]$ with $0<\beta'(x)<\infty$. As in the proof of Proposition~\ref{CNP iff state}, minimality of the action ensures that $\beta'$ is finite on all of $K_{0}(A)^+$. Extending $\beta'$ to $K_{0}(A)$ gives a $\Gamma$-invariant positive group homomorphism, $\beta$, on $K_{0}(A)$. Since $A$ is exact and projections are total, $\beta$ comes from a $\Gamma$-invariant trace on $A$, so that $A\rtimes_{\lambda}\Gamma$ admits a trace, a contradiction.

Now we assume that $S(A,\Gamma,\alpha)$ is weakly unperforated and prove $(5)\Rightarrow(1)$. Let $\theta=[x]_\alpha$ be a non-zero element in $S(A,\Gamma,\alpha)$. If $\theta$ is completely non-paradoxical then by Tarski's Theorem $S(A,\Gamma,\alpha)$ admits a non-trivial state. So, assuming $(5)$, we must have $(k+1)\theta\leq k\theta$ for some $k\in\mathbb{N}$. So
\[(k+2)\theta=(k+1)\theta+\theta\leq k\theta+\theta=(k+1)\theta\leq k\theta.\]
Repeating this trick we get $(k+1)2\theta\leq k\theta$. Since $S(A,\Gamma,\alpha)$ is weakly unperforated we conclude $2\theta\leq\theta$ and $\theta$ is properly infinite.

\end{proof}

Combining Theorems~\ref{stablyfinitecross} and~\ref{purelyinfinitecross} we obtain a dichotomy.

\begin{thm}\label{Di}Let $A$ be a unital, separable, exact C*-algebra whose projections are total. Moreover suppose $A$ has cancellation and $K_{0}(A)^+$ has the Riesz refinement property. Let $\Gamma$ be a countable discrete group and let $\alpha:\Gamma\rightarrow\Aut(A)$ be a minimal and properly outer action such that $S(A,\Gamma,\alpha)$ is weakly unperforated. Then the reduced crossed product $A\rtimes_{\lambda}\Gamma$ is a simple C*-algebra which is either stably finite or purely infinite. Moreover, if $A$ is AF and $\Gamma=\mathbb{F}_{r}$, then $A\rtimes_{\lambda}\Gamma$ is MF or purely infinite.
\end{thm}

We end our discussion with a few interesting questions. It is unknown to the author if there are examples of minimal and properly outer actions on \cstar-algebras satisfying the conditions in Theorem~\ref{Di} for which the type semigroup is \emph{not} almost unperforated. In particular, is there a free and action of the free group $\mathbb{F}_{2}$ on the Cantor set $X$ for which $S(X,\mathbb{F}_{2},\mathcal{C})$ is not almost unperforated? Although Ara and Exel construct actions of a finitely generated free group on the Cantor set for which the type semigroup is not almost unperforated,  these actions are not minimal~\cite{ArEx}. Moreover, almost unperforation may be too strong a condition to establish $(5)\Rightarrow(1)$ in Theorem~\ref{purelyinfinitecross}. What is required is that every `infinite element' (in the sense that $(k+1)x\leq kx$ for some $k$) is properly infinite. This is a priori a weaker condition than almost unperforation.

\end{document}